\documentclass[%
   , hidempi
]{mpi2015-cscpreprint}

\usepackage[american]{babel}

\usepackage{graphicx}

\usepackage{amssymb}
\usepackage{amsthm}

\usepackage{relsize}

\usepackage{bm}

\numberwithin{equation}{section}
\numberwithin{figure}{section}
\numberwithin{table}{section}

\usepackage{subcaption}
\usepackage[inline]{enumitem}

\usepackage{multirow}

\usepackage{algorithm}
\usepackage{algorithmic}

\numberwithin{algorithm}{section}

\newcommand{\INDSTATE}[1][1]{\STATE\hspace{#1\algorithmicindent}}

\newlength{\widthInput}
\newlength{\widthOutput}
\usepackage{import}

\newcommand*\xbar[1]{
	\hbox{
		\vbox{
			\hrule height 0.5pt 
			\kern0.5ex
			\hbox{
				\kern-0.35em
				\ensuremath{#1}
				\kern-0em
			}
		}
	}
}
\newcommand{\bx}{\ensuremath{\boldsymbol{{x}}}}
\newcommand{\bbx}{\ensuremath{\xbar{\boldsymbol{{\ x}}}}}

\DeclareMathOperator*{\argmax}{argmax}
\DeclareMathOperator*{\argmin}{argmin}

\DeclareMathOperator*{\range}{range}

\newtheorem{remark}{Remark}[section]

\newtheorem{prop}{Proposition}[section]
\newtheorem{theorem}{Theorem}[section]
\newtheorem{lemma}{Lemma}[section]

\theoremstyle{definition}
\newtheorem{definition}{Definition}[section]

\usepackage{combelow}

\DeclareMathOperator{\tr}{tr}

\usepackage{placeins}

\usepackage{siunitx}

\begin{document}
  

\title{Subspace-Distance-Enabled Active Learning for Efficient Data-Driven Model Reduction of Parametric Dynamical Systems}
  
\author[$\ast$]{Harshit Kapadia}
\affil[$\ast$]{Max Planck Institute for Dynamics of Complex Technical Systems, 39106 Magdeburg, Germany.\authorcr
  Corresponding author, \email{kapadia@mpi-magdeburg.mpg.de}, \orcid{0000-0003-3214-0713}}

\author[$\dagger$]{Peter Benner}
\affil[$\dagger$]{Max Planck Institute for Dynamics of Complex Technical Systems, 39106 Magdeburg, Germany.\authorcr
	\email{benner@mpi-magdeburg.mpg.de}, \orcid{0000-0003-3362-4103}}

\author[$\ddagger$]{Lihong Feng}
\affil[$\ddagger$]{Max Planck Institute for Dynamics of Complex Technical Systems, 39106 Magdeburg, Germany.\authorcr
	\email{feng@mpi-magdeburg.mpg.de}, \orcid{0000-0002-1885-3269}}
  
\shorttitle{Subspace-Distance-Enabled Active Learning} 
\shortauthor{H. Kapadia, P. Benner, L. Feng}
\shortdate{}
  

\abstract{%
	In situations where the solution of a high-fidelity dynamical system needs to be evaluated repeatedly, over a vast pool of parametric configurations and in absence of access to the underlying governing equations, data-driven model reduction techniques are preferable. We propose a novel active learning approach to build a parametric data-driven reduced-order model (ROM) by greedily picking the most important parameter samples from the parameter domain. As a result, during the ROM construction phase, the number of high-fidelity solutions dynamically grow in a principled fashion. The high-fidelity solution snapshots are expressed in several parameter-specific linear subspaces, with the help of proper orthogonal decomposition (POD), and the relative distance between these subspaces is used as a guiding mechanism to perform active learning. For successfully achieving this, we provide a distance measure to evaluate the similarity between pairs of linear subspaces with different dimensions, and also show that this distance measure is a metric. The usability of the proposed subspace-distance-enabled active learning (SDE-AL) framework is demonstrated by augmenting two existing non-intrusive reduced-order modeling approaches, and providing their active-learning-driven (ActLearn) extensions, namely, SDE-ActLearn-POD-KSNN, and SDE-ActLearn-POD-NN. Furthermore, we report positive results for two parametric physical models, highlighting the efficiency of the proposed SDE-AL approach.
}

\novelty{%
	\begin{itemize}
		\item SDE-AL: A new paradigm to perform active learning, in the context of parametric dynamical systems, to build efficient data-driven reduced-order models. 
		\item The active learning framework relies on measuring the similarity between parametric solution fields by evaluating the distance between parameter-specific POD subspaces that encompasses the respective solution behavior.
		\item A new distance metric tailored to this task is developed, which is capable of evaluating the distance between linear subspaces of different dimensions in an efficient fashion.  
		\item The proposed active learning framework is general-purpose and adaptable to any snapshot-based parametric surrogate modeling technique.
	\end{itemize}
}

\maketitle

  
\section{Introduction}%
\label{sec:intro}

Accurate numerical simulation of physical systems is crucial to predict the system dynamics, carry out control tasks, or undertake design optimization. The high-fidelity solvers based on the numerical discretization of partial differential equations (PDEs), such as the finite volume method or the finite element method, can become prohibitively expensive for high-dimensional problems. This limits their usage, particularly in scenarios requiring real-time predictions. Additionally, if the dynamical system is also parametric in nature, the situation can exacerbate when multiple queries of the high-fidelity solver for varying parameter configurations are required.

To tackle the high computational cost associated with many-query scenarios, there has been significant research and development in designing parametric model reduction strategies \cite{morQuaR14,morBenGW15,morBenOPetal17,morBenOCetal17,morHesRS16-NEW,morQuaMN16-updated,MadP21}. Among them, the reduced basis methods (RBMs) \cite{morHesRS16-NEW,morQuaMN16-updated,MadP21} aim at accelerating the computational speed for each parameter query by projecting the underlying governing equations onto a reduced basis space, thereby, requiring to solve only a reduced set of equations. The online query phase is accelerated by introducing a computationally intensive offline phase, where the reduced bases are iteratively adapted along with a sequential collection of high-fidelity solutions corresponding to important parameter locations. The parameter locations at which the high-fidelity solver is queried are typically decided by using a residual-based error estimator such that the accuracy of the resulting ROM improves in an optimal fashion during the offline phase.

For realizing intrusive model reduction techniques like the RBM, it is imperative to undertake problem-specific implementation, which can amount to a significant overhead for the user. Moreover, for dynamical systems that are primarily simulated by employing commercial software packages, it could not always be possible to access the internal implementation of the numerical solver, limiting the usability of intrusive model reduction approaches, such as the RBM. To address this issue, in recent times, there has been a considerable body of work~\cite{morIonA14,morKutBBetal16,morPehW16,morXiaFN17,morGuoH18,morHesU18,morWanHR19,morKasGH20,morYilGBetal20,MauLB21,morFreDM21-updated,FreM22,OrtDR22,McQSKetal23} proposing non-intrusive model reduction techniques that do not need access to the detailed implementation of the high-fidelity solver, but only require the data generated by it.

In the parametric setting, non-intrusive ROMs are typically built by first querying the high-fidelity solver at a fine set of parameter locations to gather the solution snapshots. Working with such a static set of solution snapshots often requires one to conservatively oversample the parameter space in order to ensure a reasonable accuracy level for the ROM solution. Unfortunately, this inflates the high-fidelity solver queries and greatly burdens the offline phase computations. In literature, limited attention has been paid to addressing this issue. Our work precisely attempts to bridge this gap for non-intrusive parametric model reduction techniques by proposing a novel active learning framework that enables an efficient adaptive sampling of the parameter space during the offline phase.

\subsection{Relation to Prior Work on Active Learning}

The existing active learning approaches in the context of data-driven model reduction for parametric systems can be broadly classified into two categories: Error-estimator-based active learning~\cite{morCheHJetal18,morZhuHBetal22,morKapFB24}, and uncertainity-minimization-based active learning~\cite{morGuoH18,morKasGH20}. The error-estimator-based active learning strategies attempt to estimate the non-intrusive ROM error across the parameter domain, thereby, picking new parameter samples in a greedy fashion, at locations corresponding to the highest estimated error values. In order to construct such an error estimator which can be queried at unseen parameter locations, an interpolant is constructed using the ROM error values corresponding to training parameter locations~\cite{morZhuHBetal22,morKapFB24}. However, the error behavior across the parameter domain is usually highly irregular, posing a significant challenge to perform accurate interpolation. This was handled in our prior work~\cite{morKapFB24} by a specific design of the error estimator and by performing interpolation using a kernel-based shallow neural network (KSNN) with the Mat\'ern class of kernels, trained via an alternating dual-staged iterative training procedure for improved accuracy. The repeated construction of the error estimator at each iteration of the active learning procedure results in an additional computational overhead in the offline phase. Moreover, without restriction to a certain class of model reduction methods~\cite{morKapFB24}, it is in general necessary to iteratively construct the ROM at every active learning iteration, which could further result in additional computational efforts.

The uncertainity-minimization-based active learning approach relies on evaluating the uncertainity associated with the ROM predictions across the parameter domain. More specifically, authors in~\cite{morGuoH18,morKasGH20} perform Gaussian process regression (GPR) over the parameter domain for approximating the reduced basis coefficients at new parameter locations. Thereby, the parameter locations associated with high variance, i.e., high standard deviation of the Gaussian process, are selected, since the prediction of the reduced basis coefficients at these locations is deemed uncertain due to a greater regression error. It is important to note that such a procedure improves the input data distribution for accurate GPR predictions, but has no connection with the underlying behavior of the function output that is being regressed. Moreover, the parametric ROM predictions in~\cite{morGuoH18,morKasGH20} are either for steady solution fields or only for a time-specific solution behavior from an unsteady dynamical system. Furthermore, this active learning strategy is inherently tied to GPR, limiting its applicability to other techniques for building non-intrusive ROMs.

There are also other noticeable works on active learning~\cite{AdcCD22,BruPV24} and reduction of parameter spaces~\cite{ConDW14,morCon15,DemTR19}. However, they do not focus on developing their ideas and applying them to the general setting of data-driven model reduction for parametric systems. Authors in~\cite{AdcCD22} provide a Christoffel adaptive sampling strategy for function approximation using deep neural networks, which can be used to approximate smooth, multivariate, output quantities of interest from parametrized PDEs by using limited data samples. Authors in~\cite{BruPV24} propose neural Galerkin schemes with active learning that allows sequential updates of a neural network parameters following the Dirac-Frenkel variational principle with an adaptive data collection strategy in time for efficient time-integration of PDEs. The active subspaces approach~\cite{ConDW14,morCon15} aims at identifying a set of important directions in the parameter/input domain by performing sensitivity analysis of any function of interest, thereby, reducing the parameter space dimensionality. Authors in~\cite{DemTR19} have used the active subspaces approach to improve the accuracy of POD modal coefficients' approximation with ridge functions. However, there is no effort to dynamically grow the training snapshot set by identification of critical locations in the parameter domain, and no notion of active learning, which is precisely the focus of our work.

\subsection{Our Contributions}

In contrast to existing approaches, we propose a new paradigm to undertake active learning, wherein important parameter locations are detected using the distance between solution subspaces. More precisely, our active learning framework relies on measuring the similarity between parameter-specific solution fields by evaluating the distance between their corresponding POD subspaces that encompasses the respective solution behavior. The active learning iterations progressively reduce the gap between parameter-specific solution subspaces. As a result, we observe that the ROM accuracy improves. It is non-trivial to compare the distance between subspaces with different dimensions. We provide a detailed discussion regarding this and also propose a new distance measure and show that it is a metric. Moreover, this distance measure can be repeatedly evaluated for numerous subspace pairs in an computationally efficient fashion, which makes it possible to be efficiently used during the active learning setup. To the best of our knowledge, such an active learning approach, referred to as the subspace-distance-enabled active learning (SDE-AL), has not been previously investigated.

The proposed SDE-AL framework can be deployed in two ways:
\begin{enumerate*}[label=(\alph*)]
	\item with a user-defined maximum feasible computational budget for parametric queries of the full-order model (FOM) solver, or
	\item with user-defined tolerance levels for the maximal subspace distance between the selected parameter locations and the maximal ROM error.
\end{enumerate*} 
For comparison with the user-specified error tolerance, the underlying ROM error can be evaluated by constructing an error estimator that provides error predictions at unseen parameter locations. In this variant of the algorithm, the error estimator is only used as a stopping criterion in the final stage of the active learning iterations. Moreover, unlike existing error-estimator-based active learning approaches, the central strategy for picking new parameter locations via SDE-AL is independent of the error estimator employed. Furthermore, the SDE-AL approach is general-purpose, meaning it can be flexibly used in conjugation with any non-intrusive reduced-order modeling method. In this article, we showcase the usability of the SDE-AL strategy for POD-KSNN \cite{morKapFB24} and POD-NN \cite{morHesU18,morWanHR19} ROMs. Their active-learning-driven counterparts are respectively referred to as SDE-ActLearn-POD-KSNN and SDE-ActLearn-POD-NN ROMs.

\subsection{Relation to Prior Work on Subspace Distance and Subspace Angles}

The usual distance between linear subspaces in the Grassmannian $Gr(p,n)$, i.e., the set of all $p$-dimensional linear subspaces in $\mathbb{R}^n$, is referred to as the geodesic distance~\cite{Wong1967}. The geodesic distance can be evaluated using the principal angles between the subspaces, which also defines a metric on the Grassmannian. There are also other notions of distance metrics on the Grassmannian~\cite{EdeAS1998,QiuXL05}. The containment gap \cite{SteS90,Kato1995} depends on the largest principal angle between subspaces. It has been used in the context of Krylov methods in~\cite{BeaES05}. Another distance metric on the Grassmannian is the chordal metric~\cite{ConHS1996,BarN02}, related to the Frobenius norm of the difference between projectors of the subspaces. In our SDE-AL framework, we are concerned with measuring the distance between linear subspaces of different dimensions. Authors in~\cite{YeL16} take an algebraic geometric view of the Grassmannian, introducing Schubert varieties, and generalizing the geodesic distance to subspaces of different dimensions. However, its evaluation necessitates explicit computation of the principal angles, requiring a singular value decomposition (SVD). During the SDE-AL procedure, we need to repeatedly evaluate the distance between numerous linear subspaces corresponding to pairs of neighboring parameter locations. Due to this, we present a distance metric that is related to the principal angles between subspaces, but does not require their explicit computation, making it computationally inexpensive and favorable for repeated evaluations. Moreover, this distance metric reduces to the chordal distance metric when the subspaces are of the same dimension, which is explained in \Cref{sec:subspace-similarity}.

In the context of reduced-order modeling, there has been prior work utilizing the subspace angles~ \cite{LieL04,LieFL05,LieF05,LieFL06,LieF07}, i.e., the principal angles between the subspaces. In their work, the authors adapt the POD basis vectors for varying parameter locations by interpolating the angles between POD subspaces. Later, authors in~\cite{AmsFL07,morAmsF08,AmsCCetal09} show an improvement in the accuracy of POD subspace adaptation by performing interpolation along the tangent space of a Grassmannian containing the set of all parameter-specific POD subspaces. All the works discussed so far focus on interpolation of POD basis vectors, and not on performing active learning for an adaptive selection of informative parameter locations. Authors in~\cite{AmsTF16} evaluate the angles between POD subspaces  to check for consistency between them, truncating the directions associated with large angles, for improving the interpolation between local parameter-specific ROMs. However, they deal with subspaces possessing the same dimension, focus on linear dynamical systems, and do not employ the subspace angles to perform active learning. An adaptive snapshot selection strategy is presented in \cite{morBenFLetal15}, where they evaluate the angle between subsequent solution snapshots along a time-trajectory to discard redundant, linearly dependent, snapshots. This can be interpreted as a parameter-specific adaptive snapshot sampling approach in the time domain. However, it cannot be used for parameter sampling.

\subsection{Organization}
\label{subsec:orga}

The remainder of the article is structured in the following fashion. \Cref{sec:subspace-similarity} provides a detailed exposition on a new strategy to measure the similarity between linear subspaces of different dimensions, which is essentially the backbone of the new active learning strategy. This is followed by outlining the complete SDE-AL framework in \Cref{sec:sde-actlearn} for constructing efficient non-intrusive parametric ROMs. Next, in \Cref{sec:existing-roms}, we present the SDE-ActLearn-POD-KSNN and SDE-ActLearn-POD-NN ROMs. \Cref{sec:num-exp} provides numerical tests detailing the performance of the proposed active learning approach. Later, we end by presenting some conclusions in \Cref{sec:conclude}.


\section{Measures of Similarity Between Linear Subspaces} \label{sec:subspace-similarity}

Consider two linear subspaces $\mathcal{X}$ and $\mathcal{Y}$ of $\mathbb{R}^n$ with their respective orthonormal bases given by the columns of matrices $X \in \mathbb{R}^{n \times p}$ and $Y \in \mathbb{R}^{n \times q}$. Consider the SVD of the matrix with inner products between the two sets of orthonormal bases, i.e., $X^\top Y = W_l \Sigma_{XY} W_r^\top$. Here, $W_l$ and $W_r$ are orthogonal matrices and $\Sigma_{XY}$ is a $p \times q$ diagonal matrix with singular values $\sigma^{XY}_1, \dots, \sigma^{XY}_s$ in decreasing order and $s = \min(p,q)$. The cosines of the principal angles between the subspaces are given by \cite{BjoG1973,ZhuK13},
\begin{equation} \label{eqn:relate-cosine-sing-val}
	\cos\boldsymbol{\Theta}(\mathcal{X},\mathcal{Y}) = \boldsymbol{\sigma}(X^\top Y) = [\sigma^{XY}_1, \dots, \sigma^{XY}_s],
\end{equation}
where $\boldsymbol{\Theta}(\mathcal{X},\mathcal{Y})$ denotes the vector of principal angles $\{\theta_k\}_{k=1}^{k=s}$ between $\mathcal{X}$ and $\mathcal{Y}$ arranged in increasing order, and $\boldsymbol{\sigma}(X^\top Y)$ denotes the vector of singular values of $X^\top Y$. We can use the $2$-norm of the vector of principal angles as a measure of similarity between two subspaces,  
\begin{equation} \label{eqn:distance-1}
	\mathcal{D}_1(\mathcal{X},\mathcal{Y}) := \| \boldsymbol{\Theta}(\mathcal{X},\mathcal{Y}) \|_2 = \| \arccos \boldsymbol{\sigma}(X^\top Y) \|_2,
\end{equation}
where $\mathcal{D}_1(\mathcal{X},\mathcal{Y})$ can be understood as a distance measure between the subspaces $\mathcal{X}$ and $\mathcal{Y}$.

To ease further discussions and assist the reader, we formally define what we mean by a metric, premetric, vector norm, and matrix norm. For more details, refer to the exposition in \cite{Rol87,HorJ85,SteS90}. 

\begin{definition}[Metric] \label{def:metric}
	Let $S$ be a set. A \emph{metric} on $S$ is a function $\rho$ of two arguments, defined on $S$, where $s_1, s_2, s_3 \in S$, which satisfies the following conditions:
	\begin{enumerate}
		\item $\rho(s_1,s_2) \ge 0$ (non-negativity),
		\item $\rho(s_1,s_2) = 0$ if and only if $s_1 = s_2$ (identity of indiscernibles),
		\item $\rho(s_1,s_2) = \rho(s_2,s_1)$ (symmetry),
		\item $\rho(s_1,s_2) \le \rho(s_1,s_3) + \rho(s_3,s_2)$ (triangle inequality).
	\end{enumerate}
\end{definition}

\begin{definition}[Premetric] \label{def:premetric}
	Let $S$ be a set. A \emph{premetric} on $S$ is a function $\rho$ of two arguments $s_1, s_2 \in S$, which satisfies the following conditions:
	\begin{enumerate}
		\item $\rho(s_1,s_2) \ge 0$ (non-negativity),
		\item $\rho(s_1,s_1) = 0$ for all $s_1 \in S$.
	\end{enumerate}
\end{definition}

\begin{definition}[Vector norm] \label{def:vector-norm}
	Let $V$ be a vector space over $\mathbb{R}$. A function $\| \cdot \|: V \rightarrow \mathbb{R}$ is a \emph{vector norm} if for all $v_1, v_2 \in V$, the following axioms are satisfied:
	\begin{enumerate}
		\item $\|v_1\| \ge 0$ (non-negative),
		\item $\|v_1\| = 0$ if and only if $v_1 = 0$ (positive),
		\item $\|c v_1\| = |c| \|v_1\|$ for all scalars $c \in \mathbb{R}$ (homogeneous),
		\item $\|v_1 + v_2\| \le \|v_1\| + \|v_2\|$ (triangle inequality).
	\end{enumerate}
\end{definition}

\begin{definition}[Matrix norm] \label{def:matrix-norm}
	Let $M$ be a space of matrices over $\mathbb{R}^{m \times n}$. A function $\| \cdot \|: M \rightarrow \mathbb{R}$ is a \emph{matrix norm} if for all $A, B \in M$ the following axioms are satisfied:
	\begin{enumerate}
		\item $\|A\| \ge 0$ (non-negative),
		\item $\|A\| = 0$ if and only if $A = 0$ (positive),
		\item $\|c A\| = |c| \|A\|$ for all scalars $c \in \mathbb{R}$ (homogeneous),
		\item $\|A + B\| \le \|A\| + \|B\|$ (triangle inequality),
		\item $\|A B\| \le \|A\| \|B\|$ (submultiplicative).
	\end{enumerate}
\end{definition}

\begin{prop} \label{prop:premetric-not-metric}
	If $\mathcal{X}$ and $\mathcal{Y}$ are linear subspaces with their respective orthonormal bases given by the columns of matrices $X \in \mathbb{R}^{n \times p}$ and $Y \in \mathbb{R}^{n \times q}$, then the distance measure $\mathcal{D}_1(\mathcal{X},\mathcal{Y})$ from \cref{eqn:distance-1} is a premetric for $p \ne q$, but in general not a metric.
\end{prop}

\begin{proof}
	Let us start by noting that we are concerned with the situation when $p \ne q$. As $\mathcal{D}_1(\cdot,\cdot)$ is computed by a vector norm, non-negativity is trivially ensured. Also, since the distance measure is defined using the principal angles between subspaces, the following holds:
	\begin{equation*}
		\mathcal{D}_1(\mathcal{X},\mathcal{Y}) = 0 \iff \mathcal{X} \subseteq \mathcal{Y} \text{ or } \mathcal{Y} \subseteq \mathcal{X}. 
	\end{equation*}
	This is a stronger statement that clearly includes the situation $\mathcal{X}=\mathcal{Y}$, which results in $\mathcal{D}_1(\mathcal{X},\mathcal{Y}) = 0$. Hence, the distance measure $\mathcal{D}_1(\cdot,\cdot)$ satisfies both the properties enlisted in \Cref{def:premetric}, and is a premetric for $p \ne q$. 
	
	Consider $\mathcal{X}_1, \mathcal{X}_2$ to be two distinct linear subspaces with respective orthonormal bases $X_1, X_2 \in \mathbb{R}^{n \times 1}$. Let $\mathcal{Y}_1$ be a linear subspace with bases $Y_1 \in \mathbb{R}^{n \times 2}$ such that $\mathcal{X}_1, \mathcal{X}_2 \subset \mathcal{Y}_1$. Assuming that the triangle inequality, property 4 in \Cref{def:metric}, is satisfied, we can write the following:
	\begin{equation*}
		\mathcal{D}_1(\mathcal{X}_1,\mathcal{X}_2) \le \mathcal{D}_1(\mathcal{X}_1,\mathcal{Y}_1) + \mathcal{D}_1(\mathcal{X}_2,\mathcal{Y}_1).
	\end{equation*} 
	Due to $\mathcal{X}_1, \mathcal{X}_2 \subset \mathcal{Y}_1$, we know that $\mathcal{D}_1(\mathcal{X}_1,\mathcal{Y}_1) = \mathcal{D}_1(\mathcal{X}_2,\mathcal{Y}_1) = 0$. Following the triangle inequality, this renders $\mathcal{D}_1(\mathcal{X}_1,\mathcal{X}_2) = 0$, which is a contradiction. This proves that $\mathcal{D}_1(\cdot,\cdot)$ does not satisfy the triangle inequality for $p \ne q$, and hence it is not a metric, but only a premetric as shown before.
\end{proof}

Note that for the example in the proof above, we considered linear subspaces of dimensions $1$ and $2$, resulting in distance evaluations between subspaces with $p=1$, $q=2$, as well as with $p=q=1$. Please refer \Cref{remark:geodesic-distance} for the scenario where all distance evaluations between linear subspaces are restricted to subspaces with $p=q$.

\begin{remark}[Geodesic distance]
	\label{remark:geodesic-distance}
	The distance measure $\mathcal{D}_1(\mathcal{X},\mathcal{Y})$ is a metric when $p=q$, and is commonly referred to as the geodesic distance~\cite{Wong1967} between linear subspaces in the Grassmannian Gr(p,n), i.e., the set of all $p$-dimensional linear subspaces in $\mathbb{R}^n$.
\end{remark}

Consider the Frobenius norm of $X^\top Y$ \cite{HorJ85,SteS90},
\begin{equation*}
	\| X^\top Y \|_F = \sqrt{\sum_{i=1}^{p}\sum_{j=1}^{q} |\mathbf{x}_i^\top \mathbf{y}_j|^2} = \sqrt{\tr((X^\top Y)^\top (X^\top Y))}.
\end{equation*}
Upon performing the SVD of $X^\top Y$, we obtain,
\begin{align*}
	\| X^\top Y \|_F &= \sqrt{\tr(W_l \Sigma_{XY} W_r^\top)^\top (W_l \Sigma_{XY} W_r^\top))}
	= \sqrt{\tr(W_r \Sigma_{XY}^\top \Sigma_{XY} W_r^\top)}	
	= \sqrt{\sum_{k=1}^{s} (\sigma^{XY}_k)^2}.
\end{align*}
Using \cref{eqn:relate-cosine-sing-val} it can be shown that $\| X^\top Y \|_F$ is related to the principal angles between the subspaces $\mathcal{X}$ and $\mathcal{Y}$,
\begin{equation} \label{eqn:frobenius-norm-angle-relation}
	\| X^\top Y \|_F = \sqrt{\sum_{k=1}^{s} (\cos \theta_k)^2}.
\end{equation}
Based on \cref{eqn:frobenius-norm-angle-relation} we are motivated to employ the following criterion to measure the similarity between subspaces,
\begin{equation} \label{eqn:distance-fro-norm}
	\tilde{\mathcal{D}}(\mathcal{X},\mathcal{Y}) := \| X^\top Y \|_F.
\end{equation}
Here, we do not need to perform the SVD of $X^\top Y$, unlike in \cref{eqn:distance-1}. 

\begin{lemma} \label{lemma:1}
	If $\mathcal{X}$ and $\mathcal{Y}$ are linear subspaces with their respective orthonormal bases given by the columns of matrices $X \in \mathbb{R}^{n \times p}$ and $Y \in \mathbb{R}^{n \times q}$, then the similarity measure $\tilde{\mathcal{D}}(\mathcal{X},\mathcal{Y})$ from \cref{eqn:distance-fro-norm} is related to the orthogonal projectors $P_\mathcal{X}$ and $P_\mathcal{Y}$ onto $\mathcal{X}$ and $\mathcal{Y}$ in the following way:
	\begin{equation} \label{eqn:lemma1-result}
		\tilde{\mathcal{D}}(\mathcal{X},\mathcal{Y}) = \sqrt{\frac{1}{2} \big( p + q - \tr[(P_\mathcal{X} - P_\mathcal{Y})^2] \big)},
	\end{equation}
	where the orthogonal projectors are defined as $P_\mathcal{X} := X X^\top$ and $P_\mathcal{Y} := Y Y^\top$.
\end{lemma}

\begin{proof} Let us start by looking at the trace operation in the right hand side of \cref{eqn:lemma1-result},
	\begin{align*}
		\tr[(P_\mathcal{X} - P_\mathcal{Y})^2] &= \tr[(X X^\top - Y Y^\top)^2] 
		\\
		&= \tr(X X^\top + Y Y^\top - X X^\top Y Y^\top - Y Y^\top X X^\top)
		\\
		&= \tr(X X^\top) + \tr(Y Y^\top) - 2 \tr(X X^\top Y Y^\top)
		\\
		&= p + q - 2 \tr(Y^\top X X^\top Y)
		\\
		&= p + q - 2 \| X^\top Y \|_F^2
		\\
		&= p + q - 2 \tilde{\mathcal{D}}^2(\mathcal{X},\mathcal{Y}).
		\\
		\implies \tilde{\mathcal{D}}(\mathcal{X},\mathcal{Y}) &= \sqrt{\frac{1}{2} \Big( p + q - \tr[(P_\mathcal{X} - P_\mathcal{Y})^2] \Big)}.
	\end{align*}
\end{proof}

Using $\tilde{\mathcal{D}}(\mathcal{X},\mathcal{Y})$, we formulate an alternate distance measure $\mathcal{D}_2(\mathcal{X},\mathcal{Y})$:
\begin{equation} \label{eqn:distance-2}
	\mathcal{D}_2(\mathcal{X},\mathcal{Y}) := \sqrt{\max(p,q) - \tilde{\mathcal{D}}^2(\mathcal{X},\mathcal{Y})}.
\end{equation}

\begin{theorem}
	If $\mathcal{X}$ and $\mathcal{Y}$ are linear subspaces with their respective orthonormal bases given by the columns of matrices $X \in \mathbb{R}^{n \times p}$ and $Y \in \mathbb{R}^{n \times q}$, then the distance measure $\mathcal{D}_2(\mathcal{X},\mathcal{Y})$ from \cref{eqn:distance-2} is a metric for all $p, q \in \mathbb{N} \backslash \{0\}$.
\end{theorem}

\begin{proof}
	Let us systematically ensure all the properties of a metric. Consider a scenario when the subspaces are identical, i.e., $\mathcal{X} = \mathcal{Y}$. This means that $p=q$ and $\range(X) = \range(Y)$. Due to the orthonormalized bases, $\tilde{\mathcal{D}}(\mathcal{X},\mathcal{X})=\sqrt{p}=\sqrt{q}$, which means $\mathcal{D}_2(\mathcal{X},\mathcal{Y})=0$. Hence, this would ensure the identity of indiscernibles, if we can also prove that $\mathcal{D}_2(\mathcal{X},\mathcal{Y})=0$ holds only when the subspaces are identical. We defer this part for later. Next, it is easy to see that the distance measure is symmetric,
	\begin{equation*}
		\mathcal{D}_2(\mathcal{X},\mathcal{Y}) = \sqrt{\max(p,q) - \| X^\top Y \|^2_F} = \sqrt{\max(q,p) - \| Y^\top X \|^2_F} = \mathcal{D}_2(\mathcal{Y},\mathcal{X}).
	\end{equation*}
	Moreover, $\mathcal{D}_2$ is non-negative by construction, i.e., $\mathcal{D}_2(\mathcal{X},\mathcal{Y}) \ge 0$.
	
	By taking the square of the distance measure $\mathcal{D}_2(\mathcal{X},\mathcal{Y})$ in \cref{eqn:distance-2} we can write,
	\begin{equation*}
		\mathcal{D}_2^2(\mathcal{X},\mathcal{Y}) = \max(p,q) - \tilde{\mathcal{D}}^2(\mathcal{X},\mathcal{Y}).
	\end{equation*}
	Substituting the form of $\tilde{\mathcal{D}}(\mathcal{X},\mathcal{Y})$ shown in \cref{eqn:lemma1-result}, leads to
	\begin{align}
		\mathcal{D}_2^2(\mathcal{X},\mathcal{Y}) &= \max(p,q) - \frac{1}{2} \Big( p + q - \tr[(P_\mathcal{X} - P_\mathcal{Y})^2] \Big). \nonumber
		\\
		\implies 2 \mathcal{D}_2^2(\mathcal{X},\mathcal{Y}) &= | p - q | + \tr[(P_\mathcal{X} - P_\mathcal{Y})^2]. \label{eqn:new-distance-square}
	\end{align}
	It is clear from \cref{eqn:new-distance-square} that $\mathcal{D}_2(\mathcal{X},\mathcal{Y}) = 0$, only when $p = q$ and the projectors $P_\mathcal{X} = P_\mathcal{Y}$, which is only possible when $\mathcal{X} = \mathcal{Y}$. 
	So far, we have successfully shown that properties $1-3$ from \Cref{def:metric} hold.
	
	Let $\mathbf{v}_p, \mathbf{v}_q \in \mathbb{R}^n$ be such that the first $p$ and $q$ elements for $\mathbf{v}_p$ and $\mathbf{v}_q$ respectively are ones and all others are zeros. This allows us to write the following,
	\begin{equation*}
		\| \mathbf{v}_p - \mathbf{v}_q \|_1 = | p - q |.\,
	\end{equation*}
	where $\|\cdot\|_1$ is the vector $1$-norm. Using the triangle inequality, we can obtain
	\begin{equation} \label{eqn:part1-triangle-inequality}
		| p - q | \le | p - r | + | r - q |, \ \forall r \ge 0.
	\end{equation}
	
	The second term in the right hand side of \cref{eqn:new-distance-square} can be rewritten as below,
	\begin{align} 
		\tr[(P_\mathcal{X} - P_\mathcal{Y})^2] &= \tr[(X X^\top - Y Y^\top)^2] \nonumber
		\\
		&= \tr[(X X^\top - Y Y^\top) (X X^\top - Y Y^\top)] \nonumber
		\\
		&= \tr[(X X^\top - Y Y^\top)^\top (X X^\top - Y Y^\top)] \nonumber
		\\
		&= \|X X^\top - Y Y^\top\|_F. \nonumber
		\\
		\implies \tr[(P_\mathcal{X} - P_\mathcal{Y})^2] &= \| P_\mathcal{X} - P_\mathcal{Y} \|_F. \label{eqn:trace-to-frobenius}
	\end{align}
	Using the triangle inequality for matrix norm, we can write
	\begin{align}
		\| P_\mathcal{X} - P_\mathcal{Y} \|_F &= \|X X^\top - Y Y^\top\|_F \nonumber
		\\
		&= \|X X^\top - Z Z^\top + Z Z^\top - Y Y^\top\|_F \nonumber
		\\
		&\le \|X X^\top - Z Z^\top\|_F + \|Z Z^\top - Y Y^\top\|_F, \nonumber
		\\
		\implies \| P_\mathcal{X} - P_\mathcal{Y} \|_F &\le \|P_\mathcal{X} - P_\mathcal{Z}\|_F + \|P_\mathcal{Z} - P_\mathcal{Y}\|_F, \label{eqn:part2-triangle-inequality}
	\end{align}
	where the columns of matrix $Z \in \mathbb{R}^{n \times r}$ correspond to the orthonormal basis spanning a linear subspace $\mathcal{Z}$.
	
	Following \cref{eqn:new-distance-square,eqn:trace-to-frobenius}, the squared distance between the subspaces $\mathcal{X}, \mathcal{Y}, \mathcal{Z}$ can be written as,
	\begin{align}
		2 \mathcal{D}_2^2(\mathcal{X},\mathcal{Y}) &= | p - q | + \| P_\mathcal{X} - P_\mathcal{Y} \|_F, \label{eqn:dist-cap-1}
		\\
		2 \mathcal{D}_2^2(\mathcal{X},\mathcal{Z}) &= | p - r | + \| P_\mathcal{X} - P_\mathcal{Z} \|_F, \label{eqn:dist-cap-2}
		\\
		2 \mathcal{D}_2^2(\mathcal{Z},\mathcal{Y}) &= | r - q | + \| P_\mathcal{Z} - P_\mathcal{Y} \|_F. \label{eqn:dist-cap-3}
	\end{align}
	where $\mathcal{D}_2(\mathcal{X},\mathcal{Z})$ and $\mathcal{D}_2(\mathcal{Z},\mathcal{Y})$ give the distance between subspaces $\mathcal{X},\mathcal{Z}$ and $\mathcal{Z},\mathcal{Y}$ respectively. By adding \cref{eqn:part1-triangle-inequality,eqn:part2-triangle-inequality}, we obtain
	\begin{equation*} 
		| p - q | + \|P_\mathcal{X} - P_\mathcal{Y}\|_F \le | p - r | + \|P_\mathcal{X} - P_\mathcal{Z}\|_F + | r - q | + \|P_\mathcal{Z} - P_\mathcal{Y}\|_F.
	\end{equation*}
	Using \cref{eqn:dist-cap-1,eqn:dist-cap-2,eqn:dist-cap-3}, this takes the following form,
	\begin{align*} 
		2 \mathcal{D}_2^2(\mathcal{X},\mathcal{Y}) &\le 2 \mathcal{D}_2^2(\mathcal{X},\mathcal{Z}) + 2 \mathcal{D}_2^2(\mathcal{Z},\mathcal{Y}).
		\\
		\implies \mathcal{D}_2(\mathcal{X},\mathcal{Y}) &\le \sqrt{\mathcal{D}_2^2(\mathcal{X},\mathcal{Z}) + \mathcal{D}_2^2(\mathcal{Z},\mathcal{Y})}.
		\\
		\implies \mathcal{D}_2(\mathcal{X},\mathcal{Y}) &\le \mathcal{D}_2(\mathcal{X},\mathcal{Z}) + \mathcal{D}_2(\mathcal{Z},\mathcal{Y}).
	\end{align*}
	This shows that the newly proposed distance measure \cref{eqn:distance-2} satisfies the triangle inequality, i.e., property $4$ in \Cref{def:metric}. 
\end{proof}

The maximal distance value $\mathcal{D}_2(\mathcal{X},\mathcal{Y})$ can render is $\sqrt{\max(p,q)}$. This corresponds to a scenario when the subspaces $\mathcal{X}$ and $\mathcal{Y}$ are orthogonal to each other, resulting in $\tilde{D} = 0$. We normalize the distance measure by its maximal value and introduce the following:
\begin{equation} \label{eqn:distance-2-normalized}
	\hat{\mathcal{D}}_2(\mathcal{X},\mathcal{Y}) = \frac{\mathcal{D}_2(\mathcal{X},\mathcal{Y})}{\sqrt{\max(p,q)}} = \sqrt{1 - \frac{\tilde{\mathcal{D}}^2(\mathcal{X},\mathcal{Y})}{\max(p,q)}},
\end{equation}
where $0 \leq \hat{\mathcal{D}}_2(\mathcal{X},\mathcal{Y}) \leq 1$. Moreover, it is important to highlight that this form of distance computation is free from the requirement to perform an SVD of the matrix with inner products between the two sets of orthonormal bases corresponding to the subspace pair, unlike in $\mathcal{D}_1(\mathcal{X},\mathcal{Y})$. This ascertains computational efficiency when repeated distance evaluations need to be made between numerous subspace pairs, as is the situation with the new active learning framework presented in this work.

\begin{remark}[Connection of $\hat{\mathcal{D}}_2(\mathcal{X},\mathcal{Y})$ with the chordal distance metric on the Grassmannian $Gr(p,n)$] 
	\label{remark:connection-with-geodesic-distance}
	When the linear subspaces $\mathcal{X}$ and $\mathcal{Y}$ have the same dimension, i.e., $p=q$, they lie on the Grassmannian $Gr(p,n)$. In this setting, recalling \cref{eqn:distance-fro-norm}, $\hat{\mathcal{D}}_2(\mathcal{X},\mathcal{Y})$ from \cref{eqn:distance-2-normalized} takes the following form:
	\begin{equation*}
		\hat{\mathcal{D}}_2(\mathcal{X},\mathcal{Y}) = \sqrt{1 - \frac{\| X^\top Y \|_F^2}{p}}. 
	\end{equation*}
	Using \cref{eqn:frobenius-norm-angle-relation}, we can write
	\begin{align}
		\hat{\mathcal{D}}_2(\mathcal{X},\mathcal{Y}) &= \sqrt{1 - \frac{\sum_{k=1}^{p} (\cos \theta_k)^2}{p}} = \frac{1}{\sqrt{p}} \sqrt{\sum_{k=1}^{p} \left( 1 - (\cos \theta_k)^2 \right)}, \nonumber
		\\
		\implies \hat{\mathcal{D}}_2(\mathcal{X},\mathcal{Y}) &= \frac{1}{\sqrt{p}} \sqrt{\sum_{k=1}^{p} (\sin \theta_k)^2} = \frac{1}{\sqrt{p}} \| \sin \boldsymbol{\Theta}(\mathcal{X},\mathcal{Y}) \|_2 = \frac{1}{\sqrt{p}} \mathcal{D}_{\mathcal{C}}(\mathcal{X},\mathcal{Y}), 
		\label{eqn:relation-with-chordal-metric}
	\end{align}
	where $\mathcal{D}_{\mathcal{C}}(\mathcal{X},\mathcal{Y}) := \| \sin \boldsymbol{\Theta}(\mathcal{X},\mathcal{Y}) \|_2$ is the chordal distance metric~\cite{ConHS1996,BarN02} on $Gr(p,n)$. We can see from \cref{eqn:relation-with-chordal-metric} that upon restricting the linear subspaces to have the same dimension, $\hat{\mathcal{D}}_2$ translates to the normalized chordal metric on $Gr(p,n)$.
\end{remark}


\section{Subspace-Distance-Enabled Active Learning Framework}
\label{sec:sde-actlearn}

In this section, we outline the core framework to actively sample new locations of high relevance in the parameter space for building a ROM by employing the distance measures described in \Cref{sec:subspace-similarity}. Let us start by denoting a parametric nonlinear dynamical system upon spatial discretization in the following manner: 
\begin{equation} \label{eqn:parametric-ode}
	\frac{d \mathbf{u}}{d t}=\mathbf{f}(\mathbf{u}, t ; \boldsymbol{\mu}), \quad \mathbf{u}(0, \boldsymbol{\mu})=\mathbf{u}_{0}(\boldsymbol{\mu}), \quad t \in[0, T],
\end{equation}
where $T \in \mathbb{R}^{+}$ denotes the final time; $\mathbf{u} \equiv \mathbf{u}(t, \boldsymbol{\mu})$ with $\mathbf{u}:[0, T] \times \mathcal{P} \rightarrow \mathbb{R}^{N}$ denotes the solution; $\mathbf{u}_{0}: \mathcal{P} \rightarrow \mathbb{R}^{N}$ denotes the parameterized initial condition; $\boldsymbol{\mu} \in \mathcal{P} \subseteq \mathbb{R}^{N_{\mu}}$ denotes the parameters; and $\mathbf{f}: \mathbb{R}^{N} \times[0, T] \times \mathcal{P} \rightarrow \mathbb{R}^{N}$ denotes a nonlinear function. In order to numerically compute the solution of \cref{eqn:parametric-ode}, the full-order model (FOM) solver can be queried at any given parameter sample. The space-time discrete solution at each parameter sample $\boldsymbol{\mu^*}$ can be represented by a snapshot matrix $U(\boldsymbol{\mu^*}) \in \mathbb{R}^{N \times (N_t + 1)}$, where each column represents a temporal instance of the solution corresponding to a fixed spatially discrete mesh. The solution snapshots are obtained along discrete time trajectories $\{ t_0, t_1, \dots , t_{N_t} \}$ with $t_0 = 0$ and $t_{N_t} = T$.

Consider that we start with $m_1 \in \mathbb{Z}^{+}$ snapshot matrices $U(\boldsymbol{\mu}_i)$ where $i \in \{1, \dots, m_1\}$, with each parameter-specific matrix having the following form,
\begin{equation}
	\label{eqn:sol-snap}
	U(\boldsymbol{\mu}_i) = [ \ \mathbf{u}(t_0,\boldsymbol{\mu}_i) \ | \ \mathbf{u}(t_1,\boldsymbol{\mu}_i) \ | \ \dots \ | \ \mathbf{u}(t_{N_t},\boldsymbol{\mu}_i) \ ].
\end{equation}
The collection of all the initial parameter samples can be represented by a set,
\begin{equation*}
	P = \{\boldsymbol{\mu}_1, \boldsymbol{\mu}_2, \dots, \boldsymbol{\mu}_{m_1}\}.
\end{equation*}
Also, consider another set $P^*$ which holds all the candidate parameter values that could be added in the set $P$ with the progression of the active learning procedure. The candidate set $P^*$ comprises of a fine sampling of the parameter domain $\mathcal{P}$, excluding the parameter values which already exist in the set $P$. It can be denoted in the following manner,
\begin{equation*}
	P^* = \{ \boldsymbol{\bar{\mu}}_1, \boldsymbol{\bar{\mu}}_2, \dots, \boldsymbol{\bar{\mu}}_{m_2} \} \ \text{with} \ \boldsymbol{\bar{\mu}}_j \neq \boldsymbol{\mu}_i; \ i=1,\dots,m_1; \ j=1,\dots,m_2.
\end{equation*}

We can now outline the key ideas behind the new active learning procedure. Let us start by finding appropriate parameter-specific linear subspaces, $\Phi(\boldsymbol{\mu}_i)$, that are representative of the snapshots collected in each $U(\boldsymbol{\mu}_i)$, where $\boldsymbol{\mu}_i \in P$. Then, we form pairs of neighboring samples in the initial parameter set P, and evaluate the distance between all these subspace pairs by employing one of the distance measures outlined in \Cref{sec:subspace-similarity}. The parameter pair whose subspaces turn out to be the farthest is selected. After which, one of the candidate parameter values, $\boldsymbol{\bar{\mu}}_j \in P^*$, that is closest to the midpoint of the selected parameter pair is picked as the new parameter sample. 

The full-order model solver is queried with this new parameter sample, which is followed by forming a linear subspace that is representative of the new snapshots. After this, the distances between the subspace corresponding to the newly picked sample with respect to its neighbor's subspaces are computed. Like before, the parameter pair whose subspaces correspond to the maximal distance is chosen, and yet another candidate sample from $P^*$ is picked such that it is closest to the bisecting point of the selected parameter pair. This procedure is repeated in an iterative fashion, either until a predefined computational budget is met, or until a user-specified distance tolerance and ROM error tolerance are satisfied. The overall active learning procedure is schematically summarized in \Cref{fig:overview}. We systematically expand on each of the aforementioned components in the remaining part of this section. 

\begin{figure}[!t]
	\centering
	\includegraphics[width=0.6\textwidth]{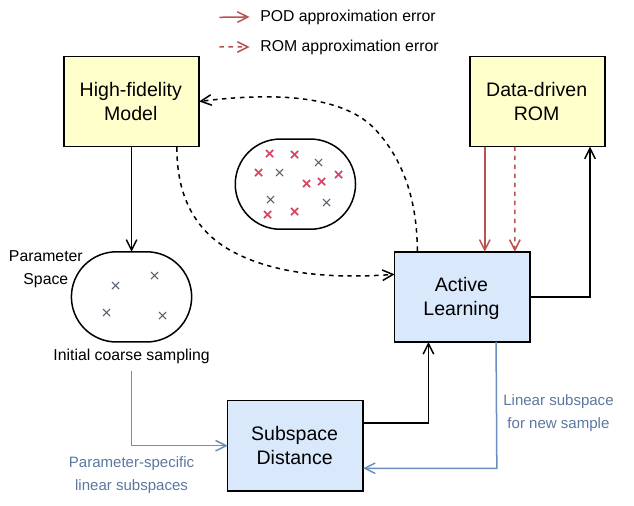}
	\caption{Overview of the subspace-distance-enabled active learning paradigm.} 
	\label{fig:overview}
\end{figure}

The columns of each parameter-specific snapshot matrix in \cref{eqn:sol-snap} span a linear subspace. We compute the bases vectors for each of these linear subspaces by performing POD of each snapshot matrix $U(\boldsymbol{\mu}_i)$, upon the following minimization:
\begin{equation} \label{eqn:pod-definition}
	\min \frac{1}{N_t + 1} \sum_{j=0}^{N_t} \left\| \mathbf{u}(t_j,\boldsymbol{\mu}_i) - \sum_{k=1}^{z} \left( \mathbf{u}(t_j,\boldsymbol{\mu}_i), \boldsymbol{\phi}_k(\boldsymbol{\mu}_i) \right)_2 \boldsymbol{\phi}_k(\boldsymbol{\mu}_i) \right\|^2_2 \ \text{with} \ z=0,\dots,N_t,
\end{equation} 
where $\left(\cdot,\cdot\right)_2$ represents the inner-product in $\mathbb{R}^N$, $\left\| \cdot \right\|_2$ represents the $2$-norm, and each $\boldsymbol{\phi}_i(\boldsymbol{\mu}_i)$ is a $\mu_i$-specific basis vector, also referred to as a POD basis vector or a POD mode. The minimization problem in \cref{eqn:pod-definition} can be solved by employing the SVD or by using the method of snapshots~\cite{Sir1987} to form the $\mu_i$-specific POD subspace. 

Consider the SVD of the snapshot matrices:
\begin{equation} \label{eqn:svd}
	U(\boldsymbol{\mu}_i) = L(\boldsymbol{\mu}_i) \Sigma(\boldsymbol{\mu}_i) K^\top(\boldsymbol{\mu}_i),
\end{equation}
where $L \in \mathbb{R}^{N \times N}$, $K \in \mathbb{R}^{(N_t+1)\times(N_t+1)}$ are orthogonal matrices with their columns as the left and right singular vectors, respectively. The matrix $\Sigma \in \mathbb{R}^{N \times (N_t+1)}$ comprises of non-negative elements on its main diagonal in a decreasing order, referred to as the singular values. We denote $\sigma_k^{(i)}$, $k = 1, \dots, s_i$ as the nonzero singular values in non-ascending order corresponding to the SVD of each snapshot matrix at $\boldsymbol{\mu}_i \in P$. The $\mu_i$-specific POD subspace is formed by selecting the first $r_i \le s_i$ left singular vectors,
\begin{equation} \label{eqn:para-linear-subspaces}
	\Phi(\boldsymbol{\mu}_i) = [ \ \mathbf{l}_1(\boldsymbol{\mu}_i) \ | \ \mathbf{l}_2(\boldsymbol{\mu}_i) \ | \ \dots \ | \ \mathbf{l}_{r_i}(\boldsymbol{\mu}_i) \ ].
\end{equation}
Here, the level of truncation $r_i$ for each $\boldsymbol{\mu}_i \in P$ is taken as the lowest value such that at least a user-defined energy criterion $\eta(\boldsymbol{\mu}_i)$, taken constant for all $\boldsymbol{\mu}_i$, is met. It is defined as,
\begin{equation} \label{eqn:energy-criterion}
	\eta(\boldsymbol{\mu}_i) = 1 - \frac{\sum_{k=1}^{k=r_i} (\sigma^{(i)}_k)^2}{\sum_{k=1}^{k=s_i} (\sigma^{(i)}_k)^2}.
\end{equation}
This means that we only use the \textit{dominant} left-singular vectors to form $\Phi(\boldsymbol{\mu}_i)$, which are known as the \textit{dominant} POD modes or POD basis vectors.

To form pairs of neighboring parameter samples in the set $P$, we can start by computing the Euclidean distance between the samples,
\begin{equation*}
	\| \boldsymbol{\mu}_m - \boldsymbol{\mu}_n \|_2 \text{, where } \boldsymbol{\mu}_m \in P, \ \boldsymbol{\mu}_n \in P, \ \boldsymbol{\mu}_m \ne \boldsymbol{\mu}_n. 
\end{equation*}
The parameter values with the least Euclidean distance for each sample $\boldsymbol{\mu}_i \in P$ can be identified, and pairs of nearest neighbors be formed. For instance, when $N_\mu = 1$, and the values in the set $P$ are ordered in an increasing fashion, each scalar parameter $\mu_i \in P$ will only have a maximum of two neighboring parameters $\mu_{i-1}$ and $\mu_{i+1}$. The minimum and maximum $\mu_i$ will only have $\mu_{i+1}$ and $\mu_{i-1}$ as the immediate neighbors, respectively. Furthermore, since the size of set $P$ is $m_1$, the total number of pairs becomes $(m_1-1)$. For $N_\mu > 1$, instead of taking the brute-force approach of checking the Euclidean distance between every possible combination of $\boldsymbol{\mu}_i \in P$, we can resort to a quicker nearest-neighbor lookup using kd-tree~\cite{FriBF1977,AryM1993,AryMNetal98,ManM02}.

\begin{figure}[!t]
	\centering
	\includegraphics[width=1\textwidth]{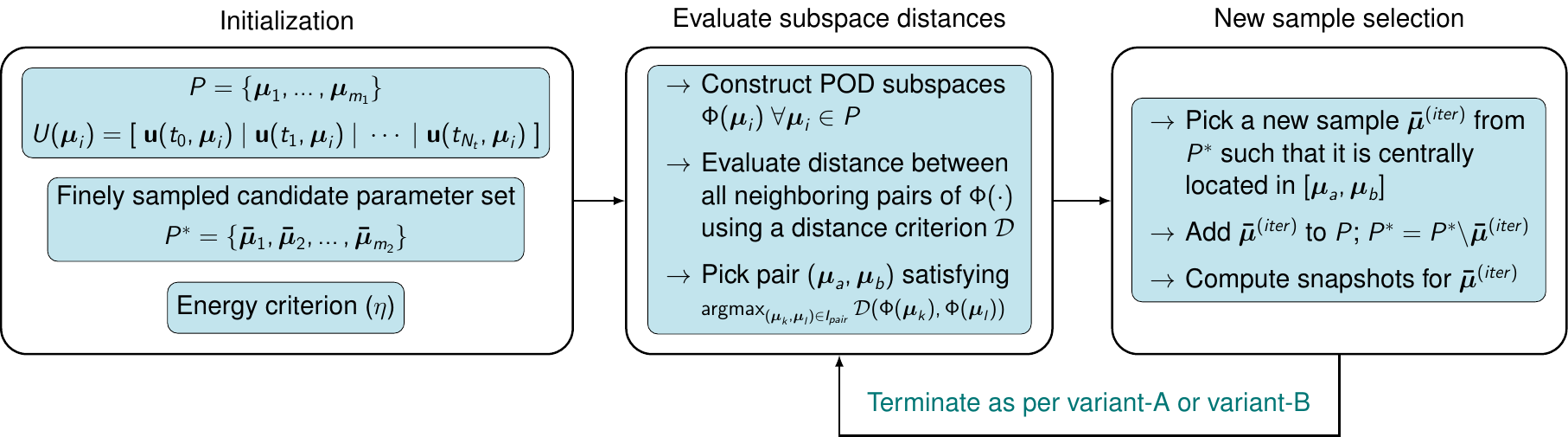}
	\caption{The SDE-AL procedure relying on iterative evaluations of the distance between parametric linear subspaces corresponding to the high-fidelity solution snapshots. The selection of new parameter samples can be concluded based on the variants proposed in \Cref{alg:sde-actlearn-A,alg:sde-actlearn-B}.} 
	\label{fig:flowchart-active-learning}
\end{figure}

Let us denote the total number of neighboring parameter pairs as $n_p \in \mathbb{Z}^+$. The set of index-pairs corresponding to all the neighboring parameter pairs can be represented in the following fashion,
\begin{equation} \label{eqn:set-index-pairs}
	I_{pair} = \{ (\boldsymbol{\mu}_k, \boldsymbol{\mu}_l) \}, \forall \boldsymbol{\mu}_k, \boldsymbol{\mu}_l \in P, \text{ such that } k \ne l \text{ and } \boldsymbol{\mu}_l \text{ is a neighbor of } \boldsymbol{\mu}_k,
\end{equation}
where the size of set $I_{pair}$ is $n_p$. For all the parameter pairs in $I_{pairs}$, we evaluate the distance between their subspaces by either following \cref{eqn:distance-1} or \cref{eqn:distance-2-normalized}, i.e.,
\begin{align}
	\mathcal{D}_1(\Phi(\boldsymbol{\mu}_k),\Phi(\boldsymbol{\mu}_l)) &= \| \arccos \boldsymbol{\sigma}(\Phi(\boldsymbol{\mu}_k)^\top \Phi(\boldsymbol{\mu}_l)) \|_2, \label{eqn:distance-1-actlearn}
	\\
	\hat{\mathcal{D}}_2(\Phi(\boldsymbol{\mu}_k),\Phi(\boldsymbol{\mu}_l)) &= \sqrt{1 - \frac{\| \Phi(\boldsymbol{\mu}_k)^\top \Phi(\boldsymbol{\mu}_l) \|_F^2}{\max(r_k,r_l)}}, \label{eqn:distance-2-actlearn}
\end{align}
where $r_k$ and $r_l$ denote the dimension of the linear subspaces represented by $\Phi(\boldsymbol{\mu}_k)$ and $\Phi(\boldsymbol{\mu}_l)$, respectively.

\begin{algorithm}[!b] 
	\caption{\texttt{Subspace-distance-enabled active learning: Variant-A} \\ {\color{darkgray} {Terminates after a user-defined computational budget for FOM solver queries is met.}}} 
	\label{alg:sde-actlearn-A}  
	\begin{algorithmic}[1]
		\REQUIRE{Initial parameter set $P$, $U(\boldsymbol{\mu}_i) \ \forall \boldsymbol{\mu}_i \in P$, candidate parameter set $P^*$, energy criterion $\eta$, {\color{darkgray} maximum permissible FOM queries $max\_query$}.}
		\ENSURE{Updated parameter set $P$, $U(\boldsymbol{\mu}_i) \ \forall \boldsymbol{\mu}_i \in P$.}
		\STATE Initialization: $iter = 1$.
		\STATE Construct $\mu_i$-specific matrix of basis vectors $\Phi(\boldsymbol{\mu}_i) \ \forall \boldsymbol{\mu}_i \in P$ by following \cref{eqn:para-linear-subspaces,eqn:energy-criterion}.
		\STATE Identify pairs of neighboring parameters in $P$ and create the set of index-pairs $I_{pair}$ by \cref{eqn:set-index-pairs}.\hspace{-2pt} 
		\STATE Evaluate distance between all $(\boldsymbol{\mu}_k, \boldsymbol{\mu}_l) \in I_{pair}$ via \cref{eqn:distance-1-actlearn} or \cref{eqn:distance-2-actlearn}. 
		\STATE Obtain the parameter pair $(\boldsymbol{\mu}_a, \boldsymbol{\mu}_b)$ whose subspaces are the farthest by following \cref{eqn:para-pair-max-distance}.
		\WHILE[Iterate until a computational budget has reached.]{$iter \le max\_query$}
		\STATE Pick a candidate sample from $P^*$ as the new sample $\boldsymbol{\bar{\mu}}^{(iter)}$ by using \cref{eqn:new-para-sample}.
		\STATE Extend $P$ by including $\boldsymbol{\bar{\mu}}^{(iter)}$. Update the candidate set, $P^* = P^* \ \backslash \ \boldsymbol{\bar{\mu}}^{(iter)}$.
		\STATE Get the solution snapshots $U(\boldsymbol{\bar{\mu}}^{(iter)})$ from a high-fidelity FOM solver or experiments.
		\STATE Construct $\Phi(\boldsymbol{\bar{\mu}}^{(iter)})$ by following \cref{eqn:para-linear-subspaces}, and decide the POD truncation level for $\boldsymbol{\bar{\mu}}^{(iter)}$ by ensuring $\eta$ is satisfied via \cref{eqn:energy-criterion}. 
		\STATE Update the pairs of neighboring parameters in $P$ and the set of index-pairs $I_{pair}$ via \cref{eqn:set-index-pairs}. 
		\STATE Evaluate the distance between any new pairs $(\boldsymbol{\mu}_k, \boldsymbol{\mu}_l) \in I_{pair}$ compared to the previous iteration by using \cref{eqn:distance-1-actlearn} or \cref{eqn:distance-2-actlearn}. 
		\STATE Obtain the parameter pair $(\boldsymbol{\mu}_a, \boldsymbol{\mu}_b)$ whose subspaces are the farthest by following \cref{eqn:para-pair-max-distance}.
		\STATE $iter = iter + 1$.
		\ENDWHILE 
	\end{algorithmic}
\end{algorithm}

The parameter pair $(\boldsymbol{\mu}_a, \boldsymbol{\mu}_b)$ from $I_{pair}$, whose subspaces turn out to be the farthest as per the distance measure \cref{eqn:distance-1-actlearn} or \cref{eqn:distance-2-actlearn}, is selected,
\begin{equation} \label{eqn:para-pair-max-distance}
	(\boldsymbol{\mu}_a, \boldsymbol{\mu}_b) = \argmax_{(\boldsymbol{\mu}_k, \boldsymbol{\mu}_l) \in I_{pair}} \mathcal{D}(\Phi(\boldsymbol{\mu}_k),\Phi(\boldsymbol{\mu}_l)).
\end{equation}
Here, $\mathcal{D}(\cdot,\cdot)$ generically refers to \cref{eqn:distance-1-actlearn} or \cref{eqn:distance-2-actlearn}. With the procedure until now, we have arrived at a pair of parameter samples from $P$ whose solution manifolds are maximally far apart. Intuitively, we would want to include a new parameter value from the region $(\boldsymbol{\mu}_a, \boldsymbol{\mu}_b)$ into the set $P$, to reduce the variability between the solutions corresponding to parameters in $P$. We propose to pick the candidate parameter value from $P^*$ that is closest to the bisecting point of the line-segment connecting $\boldsymbol{\mu}_a$ to $\boldsymbol{\mu}_b$ in the parameter domain. This can be represented in the following manner:
\begin{equation} \label{eqn:new-para-sample}
	\boldsymbol{\bar{\mu}}^{(iter)} = \argmin_{\boldsymbol{\bar{\mu}}' \in P^*} \left\lVert \frac{\boldsymbol{\mu}_a + \boldsymbol{\mu}_b}{2} - \boldsymbol{\bar{\mu}}' \right\rVert_2,
\end{equation}
where $iter$ refers to the iteration number and $\boldsymbol{\bar{\mu}}^{(iter)}$ denotes the newly selected parameter. At the start of the active learning procedure, $iter = 1$, and it is incremented by $1$ at the end of every active learning iteration.

\begin{algorithm}[!h]
	\caption{\texttt{Subspace-distance-enabled active learning: Variant-B} \\ {\color{darkgray} {Terminates after the user-defined tolerances for the subspace distance and ROM error are satisfied.}}}
	\label{alg:sde-actlearn-B}  
	\begin{algorithmic}[1]
		\REQUIRE{Initial parameter set $P$, $U(\boldsymbol{\mu}_i) \ \forall \boldsymbol{\mu}_i \in P$, candidate parameter set $P^*$, energy criterion $\eta$, {\color{darkgray} subspace distance tolerance $tol_d$, ROM error tolerance $tol_e$}.}
		\ENSURE{Updated parameter set $P$, $U(\boldsymbol{\mu}_i) \ \forall \boldsymbol{\mu}_i \in P$.}
		\STATE Initialization: $iter = 1$, $\mathcal{D}_{max}^{(iter)} = 100$, $\mathcal{E}^{(iter)} = 100$. \COMMENT{Take $\mathcal{D}_{max}^{(iter)} > tol_d$, $\mathcal{E}^{(iter)} > tol_e$.}
		\STATE Construct $\mu_i$-specific matrix of basis vectors $\Phi(\boldsymbol{\mu}_i) \ \forall \boldsymbol{\mu}_i \in P$ by following \cref{eqn:para-linear-subspaces,eqn:energy-criterion}.
		\STATE Identify pairs of neighboring parameters in $P$ and create the set of index-pairs $I_{pair}$ by \cref{eqn:set-index-pairs}.\hspace{-2pt}
		\STATE Evaluate distance between all $(\boldsymbol{\mu}_k, \boldsymbol{\mu}_l) \in I_{pair}$ via \cref{eqn:distance-1-actlearn} or \cref{eqn:distance-2-actlearn}. 
		\STATE Obtain the parameter pair $(\boldsymbol{\mu}_a, \boldsymbol{\mu}_b)$ whose subspaces are the farthest by following \cref{eqn:para-pair-max-distance}, and store $\mathcal{D}_{max}^{(iter)} = \max_{(\boldsymbol{\mu}_k, \boldsymbol{\mu}_l) \in I_{pair}} \mathcal{D}(\Phi(\boldsymbol{\mu}_k),\Phi(\boldsymbol{\mu}_l))$.
		\WHILE[Iterate until a predefined ROM error is attained.]{$\mathcal{E}^{(iter)} > tol_e$}
			\WHILE[Iterate until maximum distance reaches a level.]{$\mathcal{D}_{max}^{(iter)} > tol_d$}
				\STATE Pick a candidate sample from $P^*$ as the new sample $\boldsymbol{\bar{\mu}}^{(iter)}$ by using \cref{eqn:new-para-sample}.
				\STATE Extend $P$ by including $\boldsymbol{\bar{\mu}}^{(iter)}$. Update the candidate set, $P^* = P^* \ \backslash \ \boldsymbol{\bar{\mu}}^{(iter)}$.
				\STATE Get the solution snapshots $U(\boldsymbol{\bar{\mu}}^{(iter)})$ from a high-fidelity model or experiments.
				\STATE Construct $\Phi(\boldsymbol{\bar{\mu}}^{(iter)})$ by following \cref{eqn:para-linear-subspaces}, and decide the POD truncation level for $\boldsymbol{\bar{\mu}}^{(iter)}$ by ensuring $\eta$ is satisfied via \cref{eqn:energy-criterion}. 
				\STATE Update the pairs of neighboring parameters in $P$ and the set of index-pairs $I_{pair}$ via \cref{eqn:set-index-pairs}.\hspace{-4pt}
				\STATE Evaluate the distance between any new pairs $(\boldsymbol{\mu}_k, \boldsymbol{\mu}_l) \in I_{pair}$ compared to the previous iteration by using \cref{eqn:distance-1-actlearn} or \cref{eqn:distance-2-actlearn}. 
				\STATE Obtain the parameter pair $(\boldsymbol{\mu}_a, \boldsymbol{\mu}_b)$ whose subspaces are the farthest by following \cref{eqn:para-pair-max-distance}, and store $\mathcal{D}_{max}^{(iter)} = \max_{(\boldsymbol{\mu}_k, \boldsymbol{\mu}_l) \in I_{pair}} \mathcal{D}(\Phi(\boldsymbol{\mu}_k),\Phi(\boldsymbol{\mu}_l))$.
				\STATE $iter = iter + 1$.
			\ENDWHILE
			\STATE $\mathcal{E}^{(iter)} = \texttt{error\_estimator}(P, \ P^*, \ U(\boldsymbol{\mu}_i) \ \forall \boldsymbol{\mu}_i \in P, \ \Phi(\boldsymbol{\mu}_i) \ \forall \boldsymbol{\mu}_i \in P)$.
			\STATE $\mathcal{D}_{max}^{(iter)} = 100$. \COMMENT{$\mathcal{D}_{max}^{(iter)} > tol_d$ to allow the inner while loop to be executed.}
		\ENDWHILE \vspace{-4pt}
		\\\hrulefill
		\STATE \textbf{Function} $\texttt{error\_estimator}(P, \ P^*, \ U(\boldsymbol{\mu}_i) \ \forall \boldsymbol{\mu}_i \in P, \ \Phi(\boldsymbol{\mu}_i) \ \forall \boldsymbol{\mu}_i \in P)$ \textbf{:}
		\INDSTATE \textbf{for all} $\boldsymbol{\mu}_i \in P$ \textbf{do}
		\INDSTATE[2]$U_{POD}(\boldsymbol{\mu}_i) = \Phi(\boldsymbol{\mu}_i) \Phi^\top(\boldsymbol{\mu}_i) U(\boldsymbol{\mu}_i)$. \COMMENT{Compute POD-approximate solution.}
		\INDSTATE[2]$E(\boldsymbol{\mu}_i) = U(\boldsymbol{\mu}_i) - U_{POD}(\boldsymbol{\mu}_i)$. \COMMENT{Obtain parameter-specific error snapshots.}
		\INDSTATE[2]$\boldsymbol{\varepsilon}(\boldsymbol{\mu}_i) = \left[ \ \frac{||E_0(\boldsymbol{\mu}_i)||_2}{||U_0(\boldsymbol{\mu}_i)||_2} \ \Big| \ \frac{||E_1(\boldsymbol{\mu}_i)||_2}{||U_1(\boldsymbol{\mu}_i)||_2} \ \Big| \ \dots \ \Big| \ \frac{||E_{N_t}(\boldsymbol{\mu}_i)||_2}{||U_{N_t}(\boldsymbol{\mu}_i)||_2} \ \right]$. \COMMENT{Evaluate relative $l^2$ norm $\text{\qquad}$ of the error at each time-step $t_j \in \{ t_0, t_1, \dots , t_{N_t} \}$.}
		\INDSTATE \textbf{end for}
		\INDSTATE $\boldsymbol{\tilde{\varepsilon}}(\boldsymbol{\mu}^*) = [ \ \tilde{\varepsilon}_0(\boldsymbol{\mu}^*) \ | \ \tilde{\varepsilon}_1(\boldsymbol{\mu}^*) \ | \ \dots \ | \ \tilde{\varepsilon}_{N_t}(\boldsymbol{\mu}^*) \ ]$. \COMMENT{Create interpolant $\boldsymbol{\tilde{\varepsilon}}(\cdot)$ in parameter $\text{\quad}$ domain by using $\boldsymbol{\varepsilon}(P)$ as training data. It can be queried at any new $\boldsymbol{\mu}^*$.}
		\INDSTATE $\hat{\varepsilon}(\boldsymbol{\mu}^*) := ||\boldsymbol{\tilde{\varepsilon}}(\boldsymbol{\mu}^*)||_\infty = \max (|\tilde{\varepsilon}_0(\boldsymbol{\mu}^*)|, |\tilde{\varepsilon}_1(\boldsymbol{\mu}^*)|, \dots, |\tilde{\varepsilon}_{N_t}(\boldsymbol{\mu}^*)|)$. \COMMENT{Take $l^{\infty}$ norm in time.}
		\INDSTATE \textbf{for all} $\boldsymbol{\bar{\mu}}_j \in P^*$ \textbf{do} 
		\INDSTATE[2]Evaluate $\hat{\varepsilon}(\boldsymbol{\bar{\mu}}_j)$.
		\INDSTATE \textbf{end for} 
		\INDSTATE $\mathcal{E} = \max_{\boldsymbol{\bar{\mu}}' \in P^*} \hat{\varepsilon}(\boldsymbol{\bar{\mu}}')$.
		\INDSTATE \textbf{return} $\mathcal{E}$
	\end{algorithmic}
\end{algorithm}

The parameter set $P$ is extended by including the selected candidate parameter sample $\boldsymbol{\bar{\mu}}^{(iter)}$. At the same time, the candidate set is also updated, $P^* = P^* \backslash \boldsymbol{\bar{\mu}}^{(iter)}$. We query the full-order model solver at the new sample and obtain the snapshots $U(\boldsymbol{\bar{\mu}}^{(iter)})$. Next, POD of the new snapshot matrix is carried out, and an appropriate matrix $\Phi(\boldsymbol{\bar{\mu}}^{(iter)})$ is constructed by following \cref{eqn:para-linear-subspaces}, such that the same energy criterion $\eta$, refer \cref{eqn:energy-criterion}, is maintained for the new sample as well. The neighboring parameter pairs that correspond to the samples in the modified set $P$ are identified, which also results in a change of the total number of pairs $n_p$. At the same time, the set of index-pairs is also updated by using \cref{eqn:set-index-pairs}. For $iter > 1$, only the distances between any new pairs $(\boldsymbol{\mu}_k, \boldsymbol{\mu}_l) \in I_{pair}$ are evaluated by using either \cref{eqn:distance-1-actlearn} or \cref{eqn:distance-2-actlearn}. The procedures in \cref{eqn:para-pair-max-distance,eqn:new-para-sample}, and that of extending P, are then repeated till a stopping criterion is met.

In \Cref{alg:sde-actlearn-A,alg:sde-actlearn-B}, we present two variants to systematically terminate the iterative active learning procedure. One approach is to conclude the selection of new parameters as soon as a user-defined computational budget for the FOM solver queries is met. This quantity is denoted by $max\_query$ in \Cref{alg:sde-actlearn-A}. Another approach is to conclude the selection of new parameters as soon as the user-defined tolerances for the subspace distance and ROM error are satisfied. These tolerances are denoted by $tol_d$ and $tol_e$ in \Cref{alg:sde-actlearn-B}. The maximal distance between the subspaces is stored as $\mathcal{D}_{max}^{(iter)}$ at steps $5$ and $14$ of \Cref{alg:sde-actlearn-B}. In this variant, the iterations continue as long as $\mathcal{D}_{max}^{(iter)} > tol_d$, after which the error induced in the ROM, $\mathcal{E}^{(iter)}$, is evaluated, and the new parameter selection is finally concluded as soon as $\mathcal{E}^{(iter)} \le tol_e$. 

To estimate the error induced in the ROM, we employ the error estimator developed in \cite{morKapFB24}, which is concisely presented in steps 20--32 of \Cref{alg:sde-actlearn-B}. The error estimator is formed by constructing an interpolant between the parameter-specific POD approximation errors corresponding to all the parameters present in set $P$. This interpolant is queried at candidate parameter samples from $P^*$ and the maximal temporal error among them is regarded as the ROM error $\mathcal{E}^{(iter)}$. It is important to highlight that the active learning procedure is independent of the form of error estimator employed. Moreover, in full generality, $U_{POD}(\cdot)$ in step 23 of \Cref{alg:sde-actlearn-B} can be replaced by the solution approximation from the ROM. But this requires the ROM construction at every iteration of the outer while loop in \Cref{alg:sde-actlearn-B}. Precisely to avoid this cost, we choose to rather employ the error estimator as presented in \Cref{alg:sde-actlearn-B}, which has proven to provide a reasonable estimate for the complete ROM error in \cite{morKapFB24}.

\newpage
In \Cref{sec:num-exp}, both the subspace-distance-enabled active learning variants, as presented in \Cref{alg:sde-actlearn-A,alg:sde-actlearn-B}, are used for the numerical tests. They are numerically studied in conjugation with the model reduction methods presented in \Cref{sec:existing-roms}. In \Cref{fig:flowchart-active-learning}, a general overview of the entire active learning procedure is provided, encompassing both the proposed variants.

\begin{remark}[Regarding the candidate parameter set $P^*$ and selection of new parameter samples $\boldsymbol{\bar{\mu}}^{(iter)}$]
	\label{remark:candidate-set}
	In principle, instead of picking new parameter samples by following \cref{eqn:new-para-sample}, we can simply consider the midpoint of the line-segment connecting $\boldsymbol{\mu}_a$ to $\boldsymbol{\mu}_b$ as the new parameter sample, i.e., $\boldsymbol{\bar{\mu}}^{(iter)} = \frac{\boldsymbol{\mu}_a + \boldsymbol{\mu}_b}{2}$. However, we intentionally avoid this strategy and pick new samples from the candidate parameter set $P^*$ to avoid repeated selection of samples in a small sub-region of the parameter domain. Since $P^*$ contains some finite number of candidate parameter locations, it implicitly provides a restriction for new sample selection beyond a certain closeness in the parameter domain. As a result, whenever there are no candidate parameter locations left in the region surrounded by $\boldsymbol{\mu}_a$ and $\boldsymbol{\mu}_b$, automatically, the new sample can then be selected by using the parameter pair whose subspaces are second farthest. Such a scenario can progressively occur for any parameter pair under consideration, ultimately leading to the exhaustion of all the candidate parameter locations.
	
	Another benefit of working with the candidate parameter set $P^*$ is that it can be crafted by the user, enabling any arbitrary distribution of candidate locations across the parameter domain. This allows the possibility for using a quasi-Monte Carlo sampling strategy like the Latin hypercube sampling~\cite{McKBC1979,Tan1993,Loh1996}, to prepare a sparse distribution of candidate samples, which can be particularly useful to control the curse of dimensionality when considering a high-dimensional parameter domain. Hence,  on one hand, the SDE-AL strategy provides a mechanism that favors parameter locations minimizing the dissimilarity between parametric solution fields, and on the other hand, the user-defined distribution of candidate parameter locations in $P^*$ helps maintain a good coverage over the parameter domain. 
\end{remark}


\section{Subspace-Distance-Enabled Active-Learning-Driven ROMs}
\label{sec:existing-roms}

In this section, we present two existing non-intrusive model order reduction methods, namely POD-KSNN and POD-NN, and propose strategies to augment them using the ideas presented in \Cref{sec:sde-actlearn}. The active learning procedure provides a principled way to generate the parameter set and its corresponding high-fidelity snapshots, either as a pre-processing step prior to constructing the ROMs, or as an intermediate step during the construction of the ROMs. These two versions are respectively based on the two variants to perform subspace-distance-enabled active learning, as presented in \Cref{alg:sde-actlearn-A,alg:sde-actlearn-B}. 


\subsection{POD-KSNN ROM}
\label{subsec:pod-ksnn-rom}

The POD-KSNN method is a data-driven approach as outlined in~\cite{morKapFB24}, to create a ROM for any parametric dynamical system \cref{eqn:parametric-ode} using kernel-based shallow neural networks (KSNNs). Two interpolation steps are sequentially carried out in the parameter and time domain, with an intermediate POD reduction of the solution snapshots. On top of the fast online queries, the shallow network architecture used to create the interpolants renders a fast offline training stage.

In order to present the ROM concretely, let us first formalize the interpolation procedure, which is carried out via KSNNs. A KSNN consists of a single input, hidden, and output layer each, with radial basis kernels as its hidden layer activation functions. The width of the input layer scales with the number of data points $\{\bx_j \in \mathbb{R}^p\}_{j=1}^{\ell}$, whereas the output layer width depends on the dimension of the function we are trying to approximate. For instance, the output layer needs to have a width of $q$ for approximating the function $\mathbf{\widehat{y}}(\bbx) \in \mathbb{R}^{q}$. 

\begin{table}[!b]
	\begin{center}
		\begin{tabular}{lc}
			Name & Expression\\
			\hline\noalign{\medskip}
			Gaussian &  $e^{-(d/\epsilon)^2}$ \\
			Multi-quadric & $\sqrt{(d/\epsilon)^2 + 1}$ \\
			Cubic spline & $d^3$ \\
			\noalign{\medskip}\hline\noalign{\smallskip}
		\end{tabular}
		\caption{List of some radial basis kernels.}
		\label{tab:rbf-kernels}
	\end{center}
\end{table}

The radial basis kernel activation characterized by its width $\epsilon_i$ and center location $\mathbf{c}_i$ can be written as $\phi_i:=\phi(\| \bx - \mathbf{c}_i \|; \epsilon_i)$. For a given center location and kernel width, the activation value depends on the radial distance of the input $\bx$ from centers $\{\mathbf{c}_i\}_{i=1}^{n_c}$, i.e., $d_i = \| \bx - \mathbf{c}_i \|$, with $\|\cdot\|$ being the Euclidean norm. \Cref{tab:rbf-kernels} shows a few radial basis kernels with kernel width $\epsilon$ and radial distance $d$. The KSNN output can be represented as
\begin{align} \label{eqn:KSNN-vector-map}
	\mathbf{y}(\bx) &=  \left[ \begin{matrix} y_1(\bx) \\ y_2(\bx) \\ \vdots \\ y_q(\bx) \end{matrix} \right] = \left[ \begin{matrix} \sum_{i=1}^{n_c} w_{i}^{(1)} \phi(\| \bx - \mathbf{c}_{i} \|; \epsilon_i) \\ \sum_{i=1}^{n_c} w_{i}^{(2)} \phi(\| \bx - \mathbf{c}_{i} \|; \epsilon_i) \\ \vdots \\ \sum_{i=1}^{n_c} w_{i}^{(q)} \phi(\| \bx - \mathbf{c}_{i} \|; \epsilon_i) \end{matrix} \right], 
\end{align}
where $y_k(\bx_j) = \widehat{y}_k(\bx_j)$ is enforced for all $j=1,\ldots,\ell$ and $k=1,\ldots,q$. Thereby, the network can be trained via the alternating dual-staged iterative training (ADSIT) procedure from~\cite{morKapFB24}, either in interpolation mode ($n_c = \ell$) or in regression mode ($n_c < \ell$). Upon training, we obtain the output layer weights $\{w_{i}^{(k)}\}_{i=1}^{\ell}$ corresponding to $k=1,\ldots,q$, the kernel widths $\{\epsilon_i\}_{i=1}^{n_c}$, and/or the kernel center locations $\{\mathbf{c}_i\}_{i=1}^{n_c}$.

In the POD-KSNN method, a KSNN interpolant is first constructed using the space-time discrete solution snapshots in the parameter domain, which is referred to as $\mu$-KSNN. The trained $\mu$-KSNN can later be queried for any new parameter $\boldsymbol{\mu}^*$. For a given set of parameters $\hat{P} = \{\boldsymbol{\mu}_i\}_{i=1}^{m}$, assume that the solution snapshots are available as shown in \cref{eqn:sol-snap}. Consider each snapshot matrix $U(\boldsymbol{\mu}_i)$ reshaped as a vector in the following manner:
\begin{equation*}
	\widehat{\mathbf{y}}_u(\boldsymbol{\mu}_i) = \left[ \begin{matrix} \mathbf{u}(t_0,\boldsymbol{\mu}_i) \\ \mathbf{u}(t_1,\boldsymbol{\mu}_i) \\ \vdots \\ \mathbf{u}(t_{N_t},\boldsymbol{\mu}_i) \end{matrix} \right].
\end{equation*}
By following \cref{eqn:KSNN-vector-map}, $\mu$-KSNN is constructed by interpolating between the snapshots $\{\widehat{\mathbf{y}}_u(\boldsymbol{\mu}_i)\}_{i=1}^{m}$ in the parameter domain, which can be represented as
\begin{equation} \label{eqn:mu-KSNN-interpolant}
	\mathcal{I}^\mu(\boldsymbol{\mu}) = \left[ \begin{matrix} \mathbf{u}^I(t_0,\boldsymbol{\mu}) \\ \mathbf{u}^I(t_1,\boldsymbol{\mu}) \\ \vdots \\ \mathbf{u}^I(t_{N_t},\boldsymbol{\mu}) \end{matrix} \right] = \left[ \begin{matrix} \sum_{i=1}^{m} w_{i}^{(1)} \phi(\| \boldsymbol{\mu} - \boldsymbol{\mu}_i \|; \epsilon_i) \\ \sum_{i=1}^{m} w_{i}^{(2)} \phi(\| \boldsymbol{\mu} - \boldsymbol{\mu}_i \|; \epsilon_i) \\ \vdots \\ \sum_{i=1}^{m} w_{i}^{(N)} \phi(\| \boldsymbol{\mu} - \boldsymbol{\mu}_i \|; \epsilon_i) \ \end{matrix} \right].
\end{equation}
This allows us to write the interpolated snapshot matrix $U^I$ in the following fashion:
\begin{equation} \label{eqn:query-KSNN-surrogates}
	U^I(\boldsymbol{\mu}) = [ \ \mathbf{u}^I(t_0,\boldsymbol{\mu}) \ | \ \mathbf{u}^I(t_1,\boldsymbol{\mu}) \ | \ \dots \ | \ \mathbf{u}^I(t_{N_t},\boldsymbol{\mu}) \ ].
\end{equation}

In the online phase of POD-KSNN, $\mu$-KSNN is queried at a new parameter $\boldsymbol{\mu^*}$ and the SVD of the interpolated snapshots $U^I(\boldsymbol{\mu}^*)$ is performed, similar to \cref{eqn:svd},
\begin{equation*}
	U^I(\boldsymbol{\mu}^*) = L(\boldsymbol{\mu}^*) \Sigma(\boldsymbol{\mu}^*) K^\top(\boldsymbol{\mu}^*).
\end{equation*}
We obtain the matrix $\tilde{\Phi}(\boldsymbol{\mu}^*)$ of POD modes by collecting the truncated left singular vectors for $\boldsymbol{\mu}^*$,
\begin{equation} \label{eqn:pod-ksnn-bases}
	\tilde{\Phi}(\boldsymbol{\mu}^*) = [ \ \mathbf{l}_1(\boldsymbol{\mu}^*) \ | \ \mathbf{l}_2(\boldsymbol{\mu}^*) \ | \ \dots \ | \ \mathbf{l}_{r_*}(\boldsymbol{\mu}^*) \ ]
\end{equation}
where the truncation level $r_*$ is the lowest value such that the user-defined energy criterion $\hat{\eta}$ is satisfied, as defined below:
\begin{equation*}
	\hat{\eta} = 1 - \frac{\sum_{k=1}^{k=r_*} (\sigma_k)^2}{\sum_{k=1}^{k=s_*} (\sigma_k)^2}.
\end{equation*}

The interpolated snapshots $U^I(\boldsymbol{\mu}^*)$ can be projected on the subspace spanned by the columns of $\tilde{\Phi}(\boldsymbol{\mu}^*)$ to obtain a low-dimensional representation of the solution,
\begin{equation} \label{eqn:KSNN-sol-reduced}
	A(\boldsymbol{\mu}^*) := [ \ \alpha^0(\boldsymbol{\mu}^*) \ | \ \alpha^1(\boldsymbol{\mu}^*) \ | \ \dots \ | \ \alpha^{N_t}(\boldsymbol{\mu}^*) \ ] = \tilde{\Phi}^\top(\boldsymbol{\mu}^*) U^I(\boldsymbol{\mu}^*). 
\end{equation}
Here, $\alpha^j \in \mathbb{R}^{r_*}$ represents the reduced solution at time $t_j$. To obtain the reduced solution at any arbitrary time, an interpolant $t$-KSNN is formed in the time domain by following \cref{eqn:KSNN-vector-map},
\begin{equation} \label{eqn:query-single-KSNN-surrogate-time-domain}
	\mathcal{I}^t_{\mu^*}(t) = \left[ \begin{matrix} \sum_{i=0}^{N_t} \hat{w}_{i}^{(1)} \phi(\| t - t_i \|; \hat{\epsilon}_i) \\ \sum_{i=0}^{N_t} \hat{w}_{i}^{(2)} \phi(\| t - t_i \|; \hat{\epsilon}_i) \\ \vdots \\ \sum_{i=0}^{N_t} \hat{w}_{i}^{(r_*)} \phi(\| t - t_i \|; \hat{\epsilon}_i) \end{matrix} \right],
\end{equation}
where $\mathcal{I}^t_{\mu^*}(t_j) = \alpha^j(\boldsymbol{\mu}^*)$ is enforced for all $j = \{0, 1, \dots, N_t\}$ to train $t$-KSNN. 

The reduced solution at some new time instance $t^*$ is obtained by querying $\mathcal{I}^t_{\mu^*}$, 
\begin{equation} \label{eqn:reduced-coord-new-time}
	\alpha^I(t^*) := \mathcal{I}^t_{\mu^*}(t^*).
\end{equation}
To obtain the high-dimensional representation of the solution at $t^*$, we use the POD modes from \cref{eqn:pod-ksnn-bases},
\begin{equation} \label{eqn:sol-new-time}
	u_s(t^*, \boldsymbol{\mu}^*) = \tilde{\Phi}(\boldsymbol{\mu}^*) \ \alpha^I(t^*).
\end{equation}
We summarize the offline and online phase of the POD-KSNN ROM in \Cref{alg:POD-KSNN-offline,alg:POD-KSNN-online}.

\begin{algorithm}[!h]
	\caption{\texttt{Offline phase of POD-KSNN ROM}} 
	\label{alg:POD-KSNN-offline}  
	\begin{algorithmic}[1]
		\REQUIRE{Parameter set $\hat{P}$, snapshots $U(\boldsymbol{\mu}_i) \ \forall \boldsymbol{\mu}_i \in \hat{P}$, energy criterion $\hat{\eta}$.}
		\ENSURE{$\mu$-KSNN: $\mathcal{I}^\mu(\boldsymbol{\mu})$.}
		\STATE Construct $\mathcal{I}^\mu(\cdot)$ as in \cref{eqn:mu-KSNN-interpolant}  by building and training a KSNN.
	\end{algorithmic}
\end{algorithm}

\begin{algorithm}[!h]
	\caption{\texttt{Online phase of POD-KSNN ROM}} 
	\label{alg:POD-KSNN-online}
	\begin{algorithmic}[1]
		\REQUIRE{New parameter $\boldsymbol{\mu}^*$, new time $t^*$, energy criterion $\hat{\eta}$, $\mu$-KSNN: $\mathcal{I}^\mu(\boldsymbol{\mu})$.}
		\ENSURE{POD-KSNN solution at ($t^*$, $\boldsymbol{\mu}^*$).}
		\STATE Evaluate $U^I(\boldsymbol{\mu}^*)$ via \cref{eqn:query-KSNN-surrogates} by using $\mu$-KSNN $\mathcal{I}^\mu(\cdot)$.
		\STATE Obtain POD bases $\tilde{\Phi}(\boldsymbol{\mu}^*)$ via \cref{eqn:pod-ksnn-bases} by maintaining $\hat{\eta}$.
		\STATE Compute the reduced solution $A(\boldsymbol{\mu}^*)$ via \cref{eqn:KSNN-sol-reduced}.
		\STATE Construct $t$-KSNN $\mathcal{I}^t_{\mu^*}$ as in \cref{eqn:query-single-KSNN-surrogate-time-domain} by building and training a KSNN.
		\STATE Evaluate $\alpha^I(t^*)$ via \cref{eqn:reduced-coord-new-time}.
		\STATE Compute the POD-KSNN ROM solution $u_s(t^*, \boldsymbol{\mu}^*)$ using \cref{eqn:sol-new-time}.
	\end{algorithmic}
\end{algorithm}

\subsubsection{SDE-ActLearn-POD-KSNN ROM}

The subspace-distance-enabled active-learning-driven POD-KSNN (SDE-ActLearn-POD-KSNN) ROM can be constructed as outlined in \Cref{alg:SDE-ActLearn-POD-KSNN-offline-Variant-A,alg:SDE-ActLearn-POD-KSNN-offline-Variant-B}. The ROM is constructed by iteratively sampling new locations in the parameter domain using either \Cref{alg:sde-actlearn-A} or \Cref{alg:sde-actlearn-B} and then constructing the $\mu$-KSNN surrogate. Two variants of the algorithm---one where the ROM is constructed such that a predefined computational budget is met, the other where the ROM is constructed such that a user-defined distance and error tolerances are met---are outlined in \Cref{alg:SDE-ActLearn-POD-KSNN-offline-Variant-A,alg:SDE-ActLearn-POD-KSNN-offline-Variant-B}, respectively. The online stage of SDE-ActLearn-POD-KSNN ROM stays the same as \Cref{alg:POD-KSNN-online}.

\begin{algorithm}[!h]
	\caption{\texttt{Offline phase of SDE-ActLearn-POD-KSNN ROM: Variant-A}} 
	\label{alg:SDE-ActLearn-POD-KSNN-offline-Variant-A}  
	\begin{algorithmic}[1]
		\REQUIRE{Initial parameter set $P$, $U(\boldsymbol{\mu}_i) \ \forall \boldsymbol{\mu}_i \in P$, candidate parameter set $P^*$, energy criterion $\eta$, {\color{darkgray} maximum number of permissible iterations $max\_iter$}.}
		\ENSURE{$\mu$-KSNN surrogate $\mathcal{I}^\mu(\boldsymbol{\mu})$.}
		\STATE Obtain the updated parameter set $P$ and snapshots $U(\boldsymbol{\mu}_i) \ \forall \boldsymbol{\mu}_i \in P$ by executing \Cref{alg:sde-actlearn-A}.\hspace{-5pt}
		\STATE $\hat{P} = P$, $\hat{\eta} = \eta$.
		\STATE Construct $\mathcal{I}^\mu(\boldsymbol{\mu})$ by executing \Cref{alg:POD-KSNN-offline}.
	\end{algorithmic}
\end{algorithm}

\begin{algorithm}[!h]
	\caption{\texttt{Offline phase of SDE-ActLearn-POD-KSNN ROM: Variant-B}} 
	\label{alg:SDE-ActLearn-POD-KSNN-offline-Variant-B}  
	\begin{algorithmic}[1]
		\REQUIRE{Initial parameter set $P$, $U(\boldsymbol{\mu}_i) \ \forall \boldsymbol{\mu}_i \in P$, candidate parameter set $P^*$, energy criterion $\eta$, {\color{darkgray} subspace distance tolerance $tol_d$, ROM error tolerance $tol_e$}.}
		\ENSURE{$\mu$-KSNN surrogate $\mathcal{I}^\mu(\boldsymbol{\mu})$.}
		\STATE Obtain the updated parameter set $P$ and snapshots $U(\boldsymbol{\mu}_i) \ \forall \boldsymbol{\mu}_i \in P$ by executing \Cref{alg:sde-actlearn-B}.\hspace{-5pt}
		\STATE $\hat{P} = P$, $\hat{\eta} = \eta$.
		\STATE Construct $\mathcal{I}^\mu(\boldsymbol{\mu})$ by executing \Cref{alg:POD-KSNN-offline}.
	\end{algorithmic}
\end{algorithm}

\FloatBarrier
\subsection{POD-NN ROM}

The POD-NN method outlined in this subsection is as presented in \cite{morWanHR19} for parametrized dynamical systems. The original variant of the method first appeared in \cite{morHesU18}, which is only applicable to parametrized steady-state systems. To make the method applicable for time-dependent problems, authors in \cite{morWanHR19} resort to a two-step POD procedure. A temporal compression of the snapshots is done by performing a parameter-specific POD, which is then followed by a compression in the parameter space by applying POD to the collection of parameter-specific POD bases. The final set of dominant POD modes is then used to represent the parametric FOM snapshots in the reduced coordinates. A regression model is constructed for the reduced coordinates using a feedforward neural network, which can be queried at out-of-training time and parameter values. The high-dimensional solution at the new time and parameter value is obtained by using the predicted reduced coordinates and the final set of dominant POD modes. 

Consider that a sampling of the parameter domain is available as a set of parameters $\hat{P} = \{\boldsymbol{\mu}_i\}_{i=1}^{m}$. The solution snapshots for these parameters can be obtained from a FOM solver and can be written in the following fashion, as in \cref{eqn:sol-snap}:
\begin{equation*}
	U(\boldsymbol{\mu}_i) = [ \ \mathbf{u}(t_0,\boldsymbol{\mu}_i) \ | \ \mathbf{u}(t_1,\boldsymbol{\mu}_i) \ | \ \dots \ | \ \mathbf{u}(t_{N_t},\boldsymbol{\mu}_i) \ ],
\end{equation*}
where $i = \{1, 2, \dots, m\}$. An SVD is carried out for each parameter-specific snapshot matrix, like done in \cref{eqn:svd},
\begin{equation*}
	U(\boldsymbol{\mu}_i) = L(\boldsymbol{\mu_i}) \Sigma(\boldsymbol{\mu_i}) K^\top(\boldsymbol{\mu_i}).
\end{equation*}

We form the matrix $\Phi(\boldsymbol{\mu}_i)$ of dominant POD modes, for each parameter $\boldsymbol{\mu}_i$, by collecting the truncated left singular vectors,
\begin{equation} \label{eqn:pod-nn-bases-1}
	\Phi(\boldsymbol{\mu}_i) = [ \ \mathbf{l}_1(\boldsymbol{\mu}_i) \ | \ \mathbf{l}_2(\boldsymbol{\mu}_i) \ | \ \dots \ | \ \mathbf{l}_{r_i}(\boldsymbol{\mu}_i) \ ].
\end{equation}
Here, the truncation level $r_i$ is the lowest value such that the user-defined energy criterion $\hat{\eta}$ is satisfied:
\begin{equation} \label{eqn:pod-nn-energy-criterion-1}
	\hat{\eta} = 1 - \frac{\sum_{k=1}^{k=r_i} (\sigma^{(i)}_k)^2}{\sum_{k=1}^{k=s_i} (\sigma^{(i)}_k)^2},
\end{equation}
where $s_i$ is the total number of nonzero singular values sitting on the main diagonal of $\Sigma(\boldsymbol{\mu}_i)$.

All the parameter-specific POD modes are collected and an additional SVD is carried out as
\begin{equation} \label{eqn:pod-nn-svd-2}
	[ \ \Phi(\boldsymbol{\mu}_1) \ | \ \Phi(\boldsymbol{\mu}_2) \ | \ \dots \ | \ \Phi(\boldsymbol{\mu}_m) \ ] = \tilde{L} \tilde{\Sigma} \tilde{K}^\top,
\end{equation}
where $\tilde{L} = [ \tilde{\mathbf{l}}_1 | \tilde{\mathbf{l}}_2 | \dots | \tilde{\mathbf{l}}_N ] \in \mathbb{R}^{N \times N}$, $\tilde{K} \in \mathbb{R}^{(\sum_{i=1}^{m} r_i) \times (N_t+1)}$ are orthogonal matrices with their columns as the left and right singular vectors, respectively. The non-zero singular values in \cref{eqn:pod-nn-svd-2} are denoted by $\tilde{\sigma}_k$, $k = 1, \dots, s_{global}$. The matrix $\mathbb{V}$ representing the reduced basis is formed by selecting the first $r_{global} \le s_{global}$ dominant POD modes,
\begin{equation} \label{eqn:pod-nn-bases-2}
	\mathbb{V} = [ \ \tilde{\mathbf{l}}_1 \ | \ \tilde{\mathbf{l}}_2 \ | \ \dots \ | \ \tilde{\mathbf{l}}_{r_{global}} \ ].
\end{equation}
Here, the level of truncation $r_{global}$ is taken as the lowest value such that a user-defined energy criterion $\hat{\eta}_{global}$ is met:
\begin{equation} \label{eqn:pod-nn-energy-criterion-2}
	\hat{\eta}_{global} = 1 - \frac{\sum_{k=1}^{k=r_{global}} (\tilde{\sigma}_k)^2}{\sum_{k=1}^{k=s_{global}} (\tilde{\sigma}_k)^2}.
\end{equation}

The solution snapshots corresponding to each $\boldsymbol{\mu}_i \in \hat{P}$ are projected onto the subspace spanned by the reduced bases, so that their reduced representation is obtained,
\begin{equation*}
	B(\boldsymbol{\mu}_i) = \mathbb{V}^\top U(\boldsymbol{\mu}_i) = [ \ \beta^0(\boldsymbol{\mu}_i) \ | \ \beta^1(\boldsymbol{\mu}_i) \ | \ \dots \ | \ \beta^{N_t}(\boldsymbol{\mu}_i) \ ], 
\end{equation*}
where $\beta^j \in \mathbb{R}^{r_{global}}$ is the reduced solution at time $t_j$. To obtain the reduced solution at any arbitrary time and parameter instance, a deep feed-forward neural network surrogate is trained using the data $\{\beta^j(\boldsymbol{\mu}_i)\}_{j=0}^{N_t}$ for all $\boldsymbol{\mu}_i \in \hat{P}$, providing an NN-approximation for the reduced solution. The deep neural network surrogate generates a map $\hat{\beta}$,
\begin{align*}
	\hat{\beta}: [0, T] \times \mathcal{P} \subset \mathbb{R}^{N_\mu+1} &\rightarrow \mathbb{R}^{r_{global}},
	\\
	(t, \boldsymbol{\mu}) &\rightarrow \mathbb{V}^\top \mathbf{u}(t,\boldsymbol{\mu}).
\end{align*}
The POD-NN approximate solution $u_s$ at a new time $t^*$ and parameter $\boldsymbol{\mu}^*$ instance is obtained by,
\begin{equation*}
	u_s(t^*, \boldsymbol{\mu}^*) = \mathbb{V} \hat{\beta}(t^*, \boldsymbol{\mu}^*).
\end{equation*}
We summarize the offline and online phase of POD-NN ROM in \Cref{alg:POD-NN-offline,alg:POD-NN-online}.

\begin{algorithm}[!h]
	\caption{\texttt{Offline phase of POD-NN ROM}} 
	\label{alg:POD-NN-offline}  
	\begin{algorithmic}[1]
		\REQUIRE{Parameter set $\hat{P}$, snapshots $U(\boldsymbol{\mu}_i) \ \forall \boldsymbol{\mu}_i \in \hat{P}$, energy criteria $\hat{\eta}$ and $\hat{\eta}_{global}$.}
		\ENSURE{Reduced basis $\mathbb{V}$, NN surrogate $\hat{\beta}$.}
		\STATE Perform parameter-specific SVDs: $L(\boldsymbol{\mu_i}) \Sigma(\boldsymbol{\mu_i}) K^\top(\boldsymbol{\mu_i}) = U(\boldsymbol{\mu}_i)$.
		\STATE Determine truncation orders $r_i \ \forall \boldsymbol{\mu}_i \in \hat{P}$ using $\hat{\eta}$ via \cref{eqn:pod-nn-energy-criterion-1}.
		\STATE Create parameter-specific bases sets $\Phi(\boldsymbol{\mu}_i)$ by following \cref{eqn:pod-nn-bases-1}.
		\STATE Perform SVD by following \cref{eqn:pod-nn-svd-2}: $\tilde{L} \tilde{\Sigma} \tilde{K}^\top = [ \ \Phi(\boldsymbol{\mu}_1) \ | \ \Phi(\boldsymbol{\mu}_2) \ | \ \dots \ | \ \Phi(\boldsymbol{\mu}_m) \ ]$.
		\STATE Determine truncation order $r_{global}$ using $\hat{\eta}_{global}$ via \cref{eqn:pod-nn-energy-criterion-2}.
		\STATE Obtain the matrix $\mathbb{V}$ of reduced basis from \cref{eqn:pod-nn-bases-2}.
		\STATE Extract the reduced representation of training snapshots: $B(\boldsymbol{\mu}_i) = \mathbb{V}^\top U(\boldsymbol{\mu}_i) \ \forall \boldsymbol{\mu}_i \in \hat{P}$.
		\STATE Train a NN to obtain a regression model $\hat{\beta}$.
	\end{algorithmic}
\end{algorithm}

\begin{algorithm}[!h]
	\caption{\texttt{Online phase of POD-NN ROM}} 
	\label{alg:POD-NN-online}
	\begin{algorithmic}[1]
		\REQUIRE{New parameter $\boldsymbol{\mu}^*$, new time $t^*$, $\mathbb{V}$, NN surrogate $\hat{\beta}$.}
		\ENSURE{POD-NN solution at ($t^*$, $\boldsymbol{\mu}^*$).}
		\STATE Evaluate the output of the NN surrogate at input $(t^*, \boldsymbol{\mu}^*)$, i.e., $\hat{\beta}(t^*, \boldsymbol{\mu}^*)$.
		\STATE Compute the POD-NN solution: $u_s(t^*, \boldsymbol{\mu}^*) = \mathbb{V} \hat{\beta}(t^*, \boldsymbol{\mu}^*)$.
	\end{algorithmic}
\end{algorithm}

\subsubsection{SDE-ActLearn-POD-NN ROM}

The subspace-distance-enabled active-learning-driven POD-NN (SDE-ActLearn-POD-NN) ROM can be constructed as outlined in \Cref{alg:SDE-ActLearn-POD-NN-offline-Variant-A,alg:SDE-ActLearn-POD-NN-offline-Variant-B}. The ROM is constructed by iteratively sampling new locations in the parameter domain using either \Cref{alg:sde-actlearn-A} or \Cref{alg:sde-actlearn-B}, and then constructing the reduced basis $\mathbb{V}$ and the neural network surrogate $\hat{\beta}$ for the reduced solution. Two variants of the algorithm---one where the ROM is constructed such that a predefined computational budget is met, and the other where the ROM is constructed such that a user-defined distance and error tolerances are met---are outlined in \Cref{alg:SDE-ActLearn-POD-NN-offline-Variant-A,alg:SDE-ActLearn-POD-NN-offline-Variant-B}. The online stage of SDE-ActLearn-POD-NN ROM stays the same as \Cref{alg:POD-NN-online}.

\begin{algorithm}[!h]
	\caption{\texttt{Offline phase of SDE-ActLearn-POD-NN ROM: Variant-A}} 
	\label{alg:SDE-ActLearn-POD-NN-offline-Variant-A}  
	\begin{algorithmic}[1]
		\REQUIRE{Initial parameter set $P$, $U(\boldsymbol{\mu}_i) \ \forall \boldsymbol{\mu}_i \in P$, candidate parameter set $P^*$, energy criterion $\eta$, {\color{darkgray} maximum permissible iterations $max\_iter$}, energy criterion $\hat{\eta}_{global}$.}
		\ENSURE{Reduced basis $\mathbb{V}$, NN surrogate $\hat{\beta}$.}
		\STATE Obtain the updated parameter set $P$ and snapshots $U(\boldsymbol{\mu}_i) \ \forall \boldsymbol{\mu}_i \in P$ by executing \Cref{alg:sde-actlearn-A}.\hspace{-5pt}
		\STATE $\hat{P} = P$, $\hat{\eta} = \eta$.
		\STATE Construct $\mathbb{V}$ and surrogate $\hat{\beta}$ by executing \Cref{alg:POD-NN-offline}.
	\end{algorithmic}
\end{algorithm}

\begin{algorithm}[!h]
	\caption{\texttt{Offline phase of SDE-ActLearn-POD-NN ROM: Variant-B}} 
	\label{alg:SDE-ActLearn-POD-NN-offline-Variant-B}  
	\begin{algorithmic}[1]
		\REQUIRE{Initial parameter set $P$, $U(\boldsymbol{\mu}_i) \ \forall \boldsymbol{\mu}_i \in P$, candidate parameter set $P^*$, energy criterion $\eta$, {\color{darkgray} subspace distance tolerance $tol_d$, ROM error tolerance $tol_e$}, energy criterion $\hat{\eta}_{global}$.}
		\ENSURE{Reduced basis $\mathbb{V}$, NN surrogate $\hat{\beta}$.}
		\STATE Obtain the updated parameter set $P$ and snapshots $U(\boldsymbol{\mu}_i) \ \forall \boldsymbol{\mu}_i \in P$ by executing \Cref{alg:sde-actlearn-B}.\hspace{-5pt}
		\STATE $\hat{P} = P$, $\hat{\eta} = \eta$.
		\STATE Construct $\mathbb{V}$ and surrogate $\hat{\beta}$ by executing \Cref{alg:POD-NN-offline}.
	\end{algorithmic}
\end{algorithm}

\FloatBarrier

\section{Numerical Results}%
\label{sec:num-exp}

We demonstrate the applicability of the SDE-AL framework for two parameter-dependent dynamical systems. The first system corresponds to a problem setup featuring fluid flow behavior governed by the shallow water equations, which is challenging due to the numerous shock interactions across the spatial domain as the system evolves in time. The nature of these interactions is affected by the fluid viscosity. Next, we consider a problem showcasing a fluid flow with a vortex shedding phenomena governed by the incompressible Navier-Stokes equations. Here, the shedding behavior depends upon the Reynolds number under consideration. 

For brevity, we present the results from the SDE-ActLearn-POD-KSNN ROM for the shallow water equations, whereas, the results from the SDE-ActLearn-POD-NN ROM are presented for the Navier-Stokes equations. Moreover, variant-A of the SDE-AL framework, i.e., \Cref{alg:sde-actlearn-B}, is tested for the shallow water equations example, whereas, variant-B of the SDE-AL framework, i.e., \Cref{alg:sde-actlearn-A}, is tested for the Navier-Stokes equations example. Furthermore, for the shallow water equations example, we also provide a comparison between the SDE-AL results obtained while using $\mathcal{D}_1$ and those while using $\hat{\mathcal{D}}_2$. This comparison highlights the benefit of using the newly proposed distance metric for performing SDE-AL.

The rest of this section outlines both the problem setups in detail and provides the results obtained from the numerical experiments. For details about the parametric interpolant structure ($\mu$-KSNN) and the temporal interpolant structure ($t$-KSNN) for SDE-ActLearn-POD-KSNN, as well as the NN structure for SDE-ActLearn-POD-NN, please refer \Cref{appendix:implementation-details}.  

\subsection{Shallow Water Equations}%
\label{subsec:swe}

The shallow water equations are typically used to model the behavior of water bodies like a lake or canal, as well as in weather forecasting~\cite{CouG1988,Cha04}. We consider them in the following parametric form:
\begin{align*}
	\partial_t h + \partial_x (h u) &= 0, \\
	\partial_t(h u) + \partial_x \Big( h u^2 + \frac{1}{2} g h^2 \Big) &= - \frac{\nu}{\lambda} u,
\end{align*}
where $h(t,x)$ is the fluid height, $u(t,x)$ is the depth-averaged velocity, $g$ is the gravitational acceleration, $\nu$ is the fluid's dynamic viscosity, and $\lambda$ is the fluid's mean-free path. The initial fluid height and velocity are prescribed as follows:
\begin{align*}
	h(0,x) &= 1 + \exp(3 \cos(\pi (x + 0.5)) - 4),
	\\
	u(0,x) &= 0.25.
\end{align*}

The problem setup and the FOM solver are kept similar to~\cite{morKapFB24}, allowing us to easily compare the results of the two active learning approaches. The system of equations is solved using the discontinuous Galerkin method, where the local Lax–Friedrichs Riemann solver is employed to evaluate the numerical flux, and a second-order strong stability preserving Runge–Kutta scheme is used for time integration. We fix $\lambda=0.1$ and consider $\nu$ as the free parameter, ranging from $1$ to $10^{-5}$. Furthermore, we consider $100$ viscosity values in total, accounting for both the parameter sets $P$ and $P^*$. The parameter samples considered during active learning are $100$ evenly distributed logarithmic $\nu$ values in $[\log_{10}(10^{-5}), \log_{10}(1)]$.  The spatial domain $\Omega \in [-1, 1]$ has $601$ grid nodes and is equipped with periodic boundary conditions. FOM solutions corresponding to each training parameter sample are collected at $200$ uniform time steps in $t \in [0, 2]$, creating parameter-specific snapshot matrices.

\begin{figure}[!b] %
	\centering
	\begin{subfigure}[b]{0.35\textwidth}
			\centering
			\includegraphics[width=1\textwidth, trim=0 0 0 0, clip]{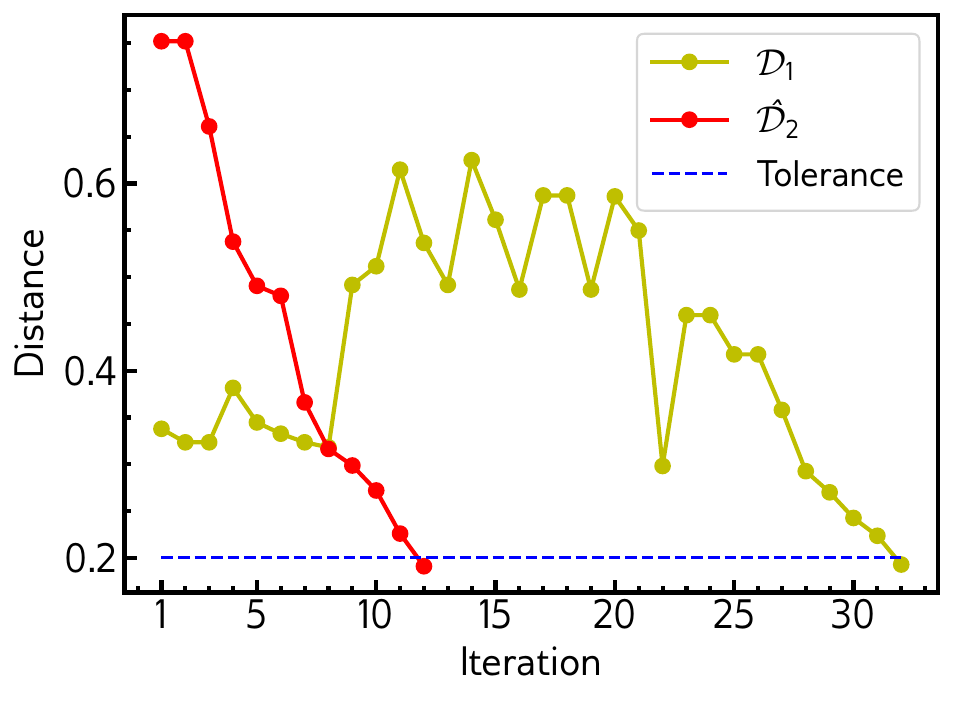}
			\caption{}
			\label{fig:swe-active-learn-a}
	\end{subfigure}
	\hspace{2em}
	\begin{subfigure}[b]{0.35\textwidth}
			\centering
			\includegraphics[width=1\textwidth, trim=0 0 0 0, clip]{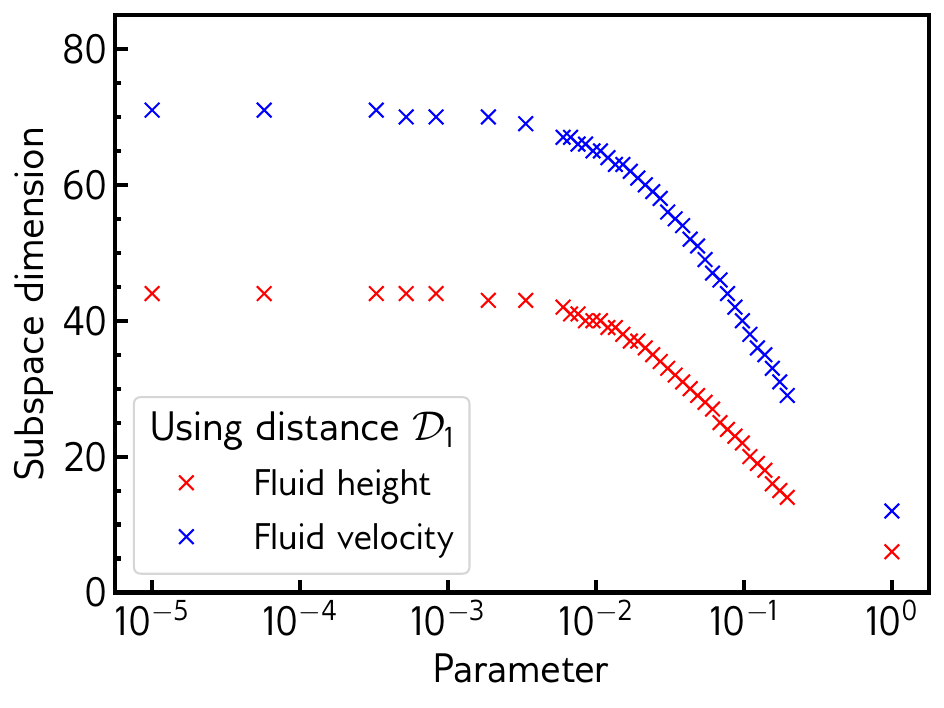}
			\caption{}
			\label{fig:swe-active-learn-b}
	\end{subfigure}
	\hspace{2em}
	\begin{subfigure}[b]{0.35\textwidth}
		\centering
		\includegraphics[width=1\textwidth, trim=0 0 0 0, clip]{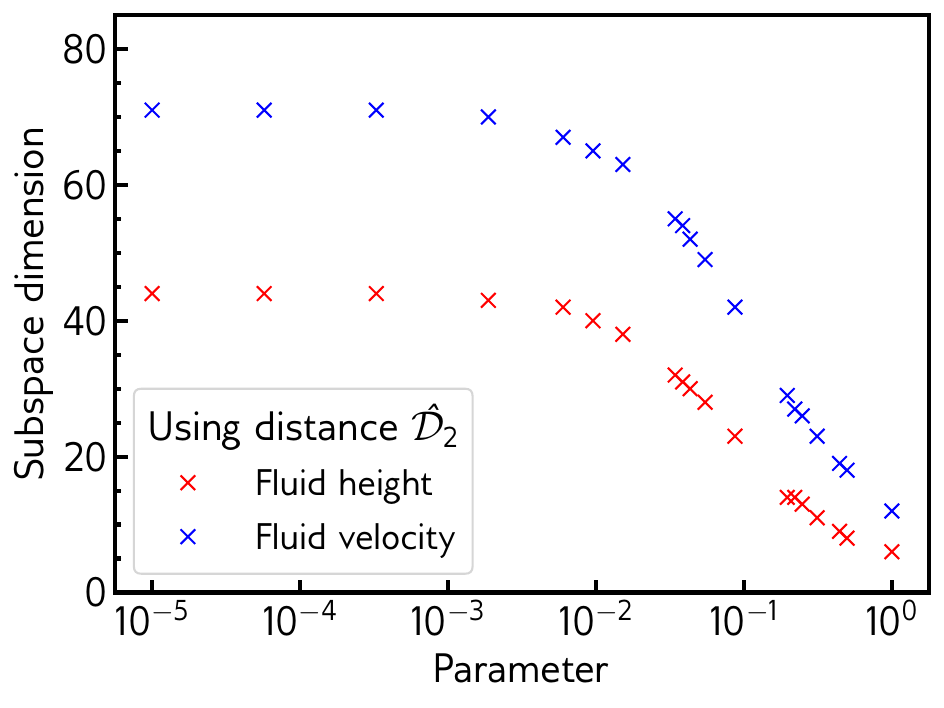}
		\caption{}
		\label{fig:swe-active-learn-c}
	\end{subfigure}
	\caption{Shallow water equations: Comparison between subspace-distance-enabled active learning procedures with distance criteria $\mathcal{D}_1$ and $\hat{\mathcal{D}}_2$. (a)~shows the variation in the maximal distance value  $\mathcal{D}_{max}^{(iter)}$ with iterations. For better readability, and to easily distinguish between the two cases, $\mathcal{D}_1$ and $\hat{\mathcal{D}}_2$ are used in the legend. (b)~and~(c)~show the selected parameter samples along with the dimensions of their respective subspace pairs, obtained by employing $\mathcal{D}_1$ and $\hat{\mathcal{D}}_2$, respectively.}
	\label{fig:swe-active-learn}
\end{figure}

The SDE-ActLearn-POD-KSNN ROM is constructed as outlined in \Cref{alg:SDE-ActLearn-POD-KSNN-offline-Variant-B}. The active learning procedure is initiated by picking $8$ viscosity values, with the following indices:
\begin{equation*}
	\{0,99,15,30,45,55,70,85\}.
\end{equation*} 
Here, the index starts from $0$ when counting the $100$ viscosity values ordered in an ascending fashion. The set $P$ is initialized by these $8$ parameter samples. For each of these viscosity values, a snapshot matrix is constructed by querying the FOM solver. Furthermore, using the snapshot matrices, parameter-specific subspaces are formed, retaining $99.9999\%$ of the total energy, which corresponds to $\eta = 10^{-6}$.

The distances between all neighboring parameter-specific subspaces are evaluated and the parameter pair with the most distant subspaces is identified, followed by the selection of a new sample from the candidate set $P^*$ as outlined in \Cref{alg:sde-actlearn-B}. This process is continued iteratively by updating the neighboring parameter pairs, evaluating the distance between new pairs, and selecting new samples from $P^*$ in parameter regions surrounded by the subspace pair that are farthest away. 

We compare two active learning procedures following \Cref{alg:sde-actlearn-B}, by using $\mathcal{D}_1$ and $\hat{\mathcal{D}}_2$. When using $\mathcal{D}_1$, the variation of maximal distance value  $\mathcal{D}_{max}^{(iter)}$ with iterations is quite erratic and takes $31$ parameter selections to attain a distance tolerance $tol = 0.2$, and an error tolerance $tol = 10^{-2}$. Whereas, with distance criterion $\hat{\mathcal{D}}_2$, the maximal distance is seen to monotonically decrease, also at a faster rate. Moreover, we notice that the SDE-AL procedure employing $\hat{\mathcal{D}}_2$ converges with just $11$ new parameter selections.

\Cref{fig:swe-active-learn-b,fig:swe-active-learn-c} show the actively sampled parameters and the dimension of their corresponding subspaces. We observe that in regions where the solution behavior changes significantly, relatively more samples are being selected, which is also accompanied by a considerable variation in the POD subspace dimension. From \Cref{fig:swe-active-learn-b,fig:swe-active-learn-c}, we can also deduce that while using $\mathcal{D}_1$, as opposed to $\hat{\mathcal{D}}_2$, apart from oversampling of the parameter space, the sampling pattern is also slightly different.

\begin{figure}[!b] 
	\centering
	
	\includegraphics[width=0.47\textwidth, trim=0 0 0 0, clip]{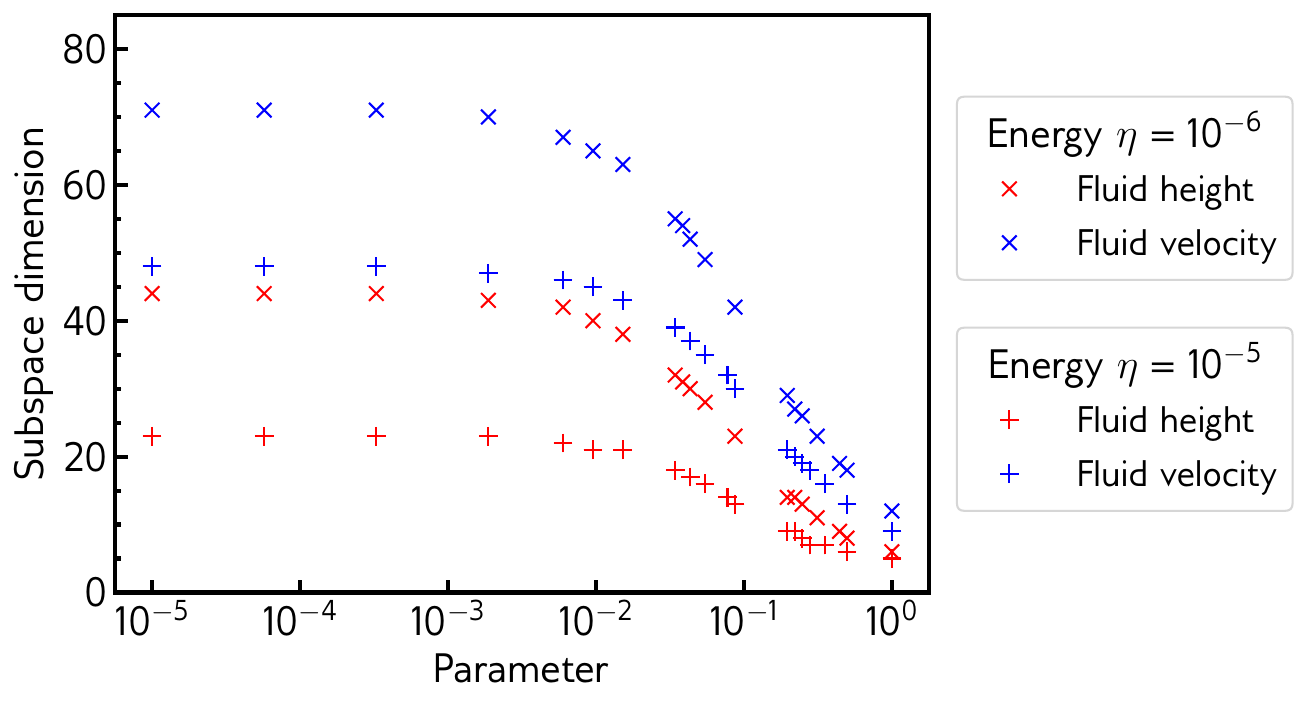}
	
	\caption{Shallow water equations: Comparison of the parameter selection and their respective POD subspace dimension upon performing subspace-distance-enabled active learning with different values of energy criterion $\eta$.}
	
	\label{fig:swe-active-learn-energy-compare}
\end{figure}

The SDE-ActLearn-POD-KSNN ROMs are queried at the following out-of-training viscosity values: $\{5 \times 10^{-1}, 5 \times 10^{-2}, 5 \times 10^{-3}, 5 \times 10^{-4}, 5 \times 10^{-5}\}$. By querying the ROM constructed using the distance criterion $\hat{\mathcal{D}}_2$, \Cref{fig:swe-h-rom-sol-surf,fig:swe-u-rom-sol-surf} show the fluid height and velocity predictions at three representative test samples, along with their respective relative solution errors, in the entire space-time domain. The reported relative error values are the absolute difference between the FOM and ROM solutions, normalized by the $2$-norm of the FOM solution. It is clear from \Cref{fig:swe-h-rom-sol-surf,fig:swe-u-rom-sol-surf} that the ROM is able to generate very accurate solution profiles, successfully capturing the multiple shock interactions over time. A similar solution accuracy is observed for the SDE-ActLearn-POD-KSNN ROM constructed using the criterion $\mathcal{D}_1$. As a result, to avoid similar figures, we do not show the solution and error contour plots for it. 

Due to a sparser sample selection with a similar level of accuracy, it is computationally beneficial to employ the distance measure $\hat{\mathcal{D}}_2$ during the SDE-AL procedure. Moreover, it is also crucial to note that while using $\hat{\mathcal{D}}_2$, the SDE-AL procedure is not sensitive to the choice of the energy criterion. This is evident from \Cref{fig:swe-active-learn-energy-compare} which compares the SDE-AL parameter selection using $\hat{\mathcal{D}}_2$ corresponding to two distinct energy criterion values, namely, $10^{-5}$ and $10^{-6}$. This choice affects the parameter-specific subspaces that are used to evaluate the (dis)similarity between the parametric solution snapshots. We notice that relaxing the amount of energy retained in the parameter-specific subspaces to $10^{-5}$ reduces the dimension of the parameter-specific subspaces, but still results in a near identical sampling of the parameter space. Due to this robustness of the sample selection with regards to the amount of information retained in the subspaces, we can in principle work with a set of compact subspaces and still obtain a similar parameter sampling.

\begin{figure}[!t]
	\centering
	\begin{subfigure}[b]{0.32\textwidth}
			\centering
			\includegraphics[width=.95\textwidth, trim=0 0 0 70, clip]{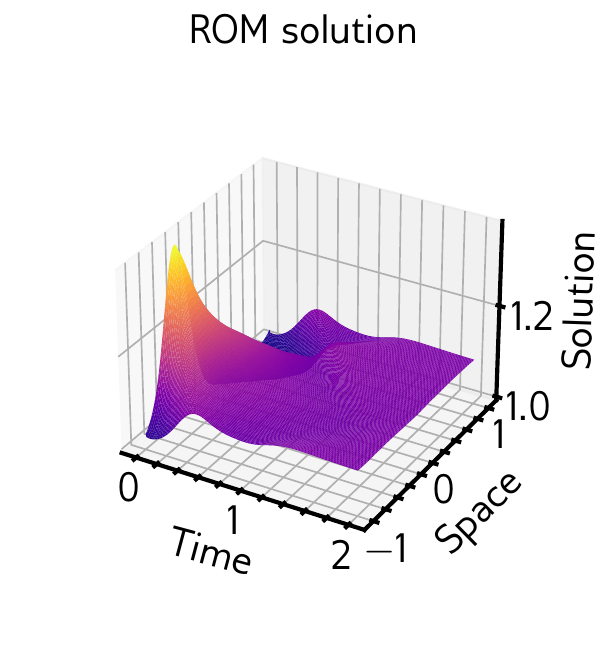}
	\end{subfigure}
	\begin{subfigure}[b]{0.32\textwidth}
		\centering
		\includegraphics[width=0.95\textwidth, trim=0 0 0 70, clip]{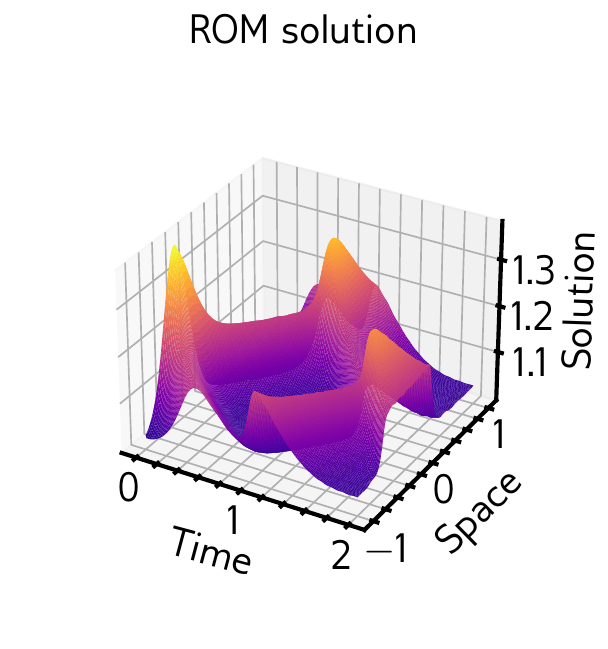}
	\end{subfigure}
	\begin{subfigure}[b]{0.32\textwidth}
		\centering
		\includegraphics[width=0.95\textwidth, trim=0 0 0 70, clip]{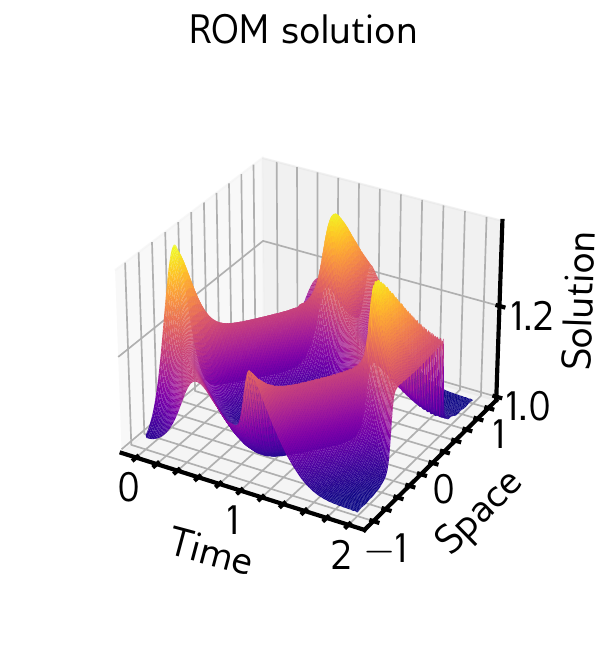}
	\end{subfigure}
	\begin{subfigure}[b]{0.31\textwidth}
			\centering
			\includegraphics[width=0.85\textwidth, trim=0 0 0 40, clip]{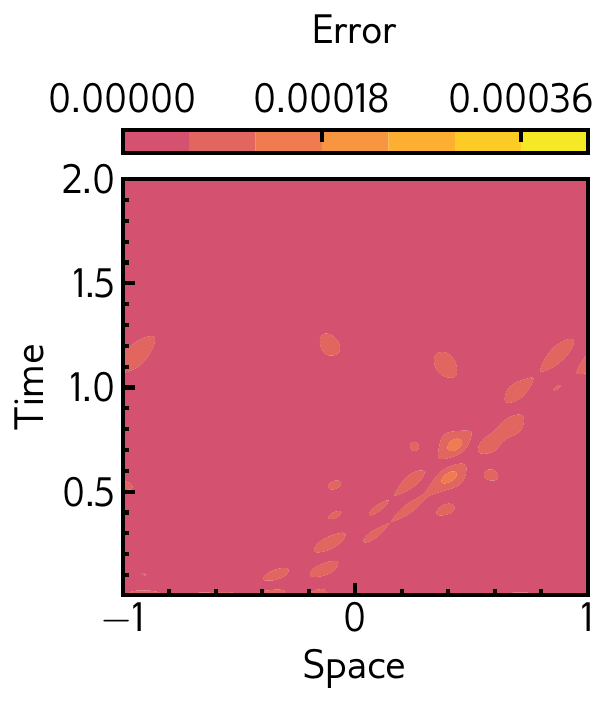}
	\end{subfigure}
	\begin{subfigure}[b]{0.31\textwidth}
			\centering
			\includegraphics[width=0.85\textwidth, trim=0 0 0 40, clip]{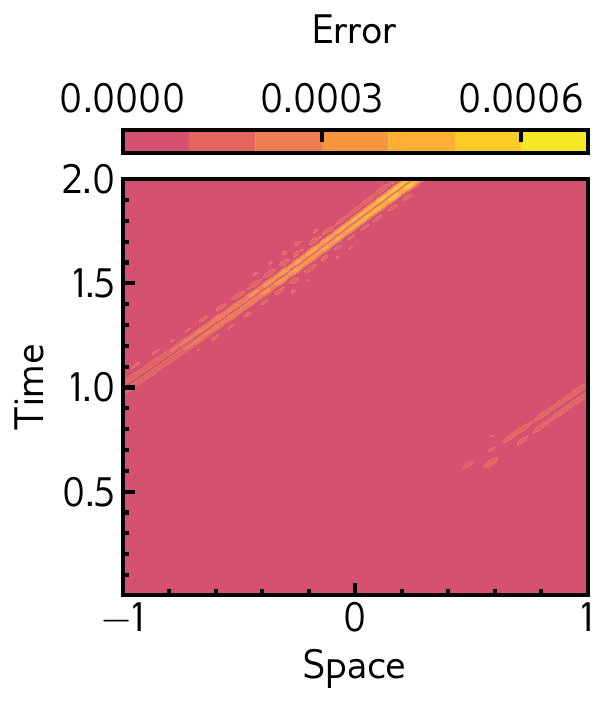}
	\end{subfigure}
	\begin{subfigure}[b]{0.31\textwidth}
			\centering
			\includegraphics[width=0.85\textwidth, trim=0 0 0 40, clip]{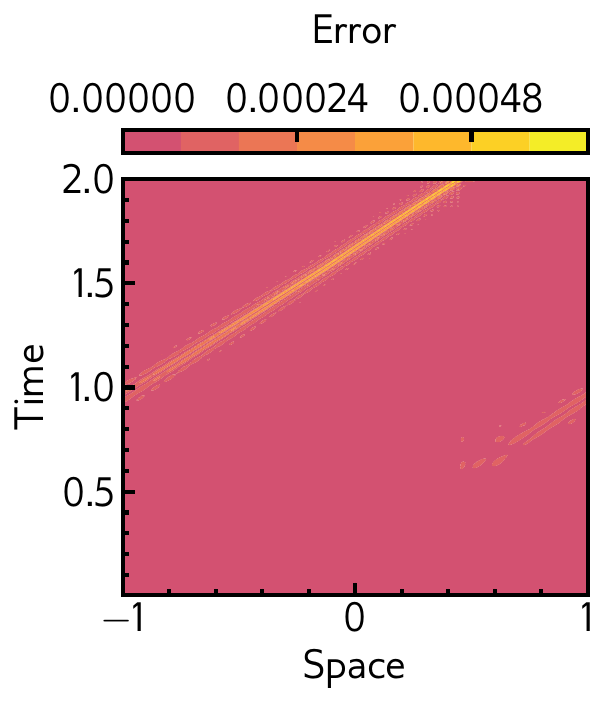}
	\end{subfigure}
	\caption{Shallow water equations' fluid height: The SDE-ActLearn-POD-KSNN solution (shown in first row), and the relative solution error (shown in second row). The viscosity values $\nu$ going from left to right are in the following order: $\{5 \times 10^{-1}, 5 \times 10^{-2}, 5 \times 10^{-5}\}$. All these $\nu$ values and time instances are outside of the training set. The active learning process is carried out using the distance criterion $\hat{\mathcal{D}}_2$.} 
	\label{fig:swe-h-rom-sol-surf}
\end{figure}

\begin{figure}[!t] 
	\centering
	\begin{subfigure}[b]{0.32\textwidth}
			\centering
			\includegraphics[width=0.95\textwidth, trim=0 0 0 70, clip]{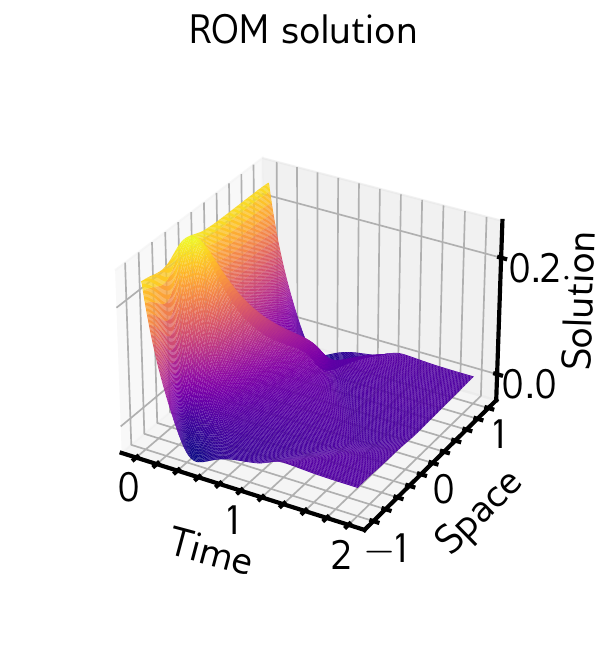}
	\end{subfigure}
	\begin{subfigure}[b]{0.32\textwidth}
		\centering
		\includegraphics[width=0.95\textwidth, trim=0 0 0 70, clip]{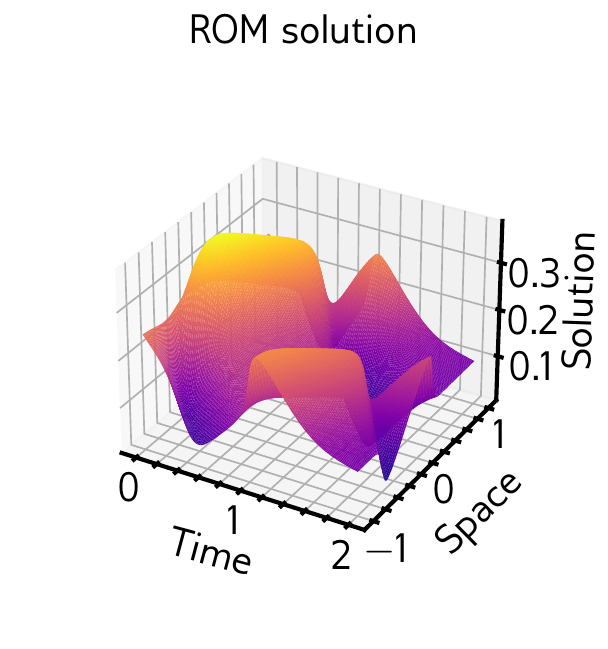}
	\end{subfigure}
	\begin{subfigure}[b]{0.32\textwidth}
		\centering
		\includegraphics[width=0.95\textwidth, trim=0 0 0 70, clip]{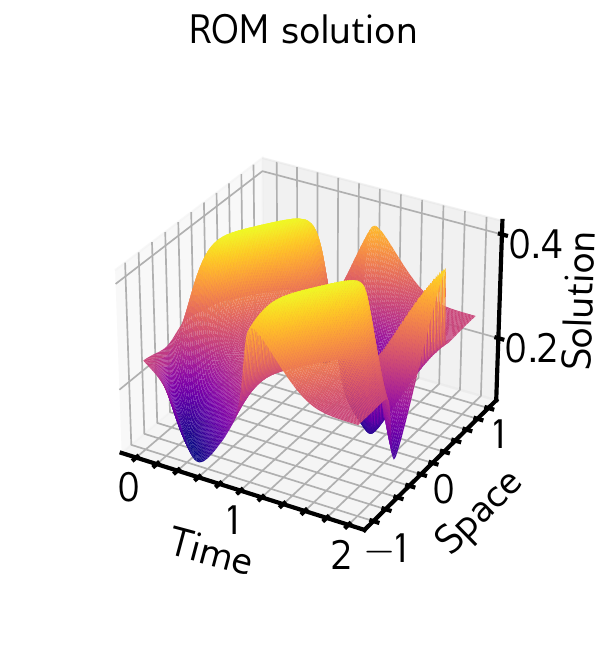}
	\end{subfigure}
	\begin{subfigure}[b]{0.31\textwidth}
			\centering
			\includegraphics[width=0.85\textwidth, trim=0 0 0 40, clip]{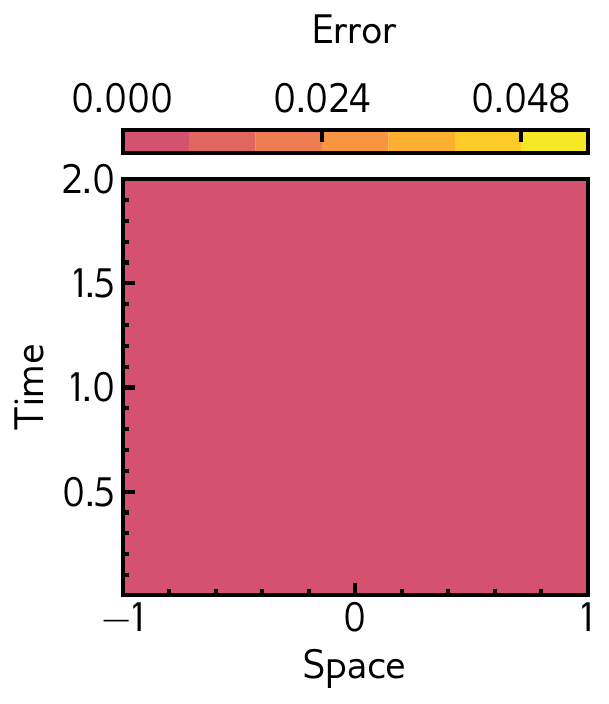}
	\end{subfigure}
	\begin{subfigure}[b]{0.31\textwidth}
			\centering
			\includegraphics[width=0.85\textwidth, trim=0 0 0 40, clip]{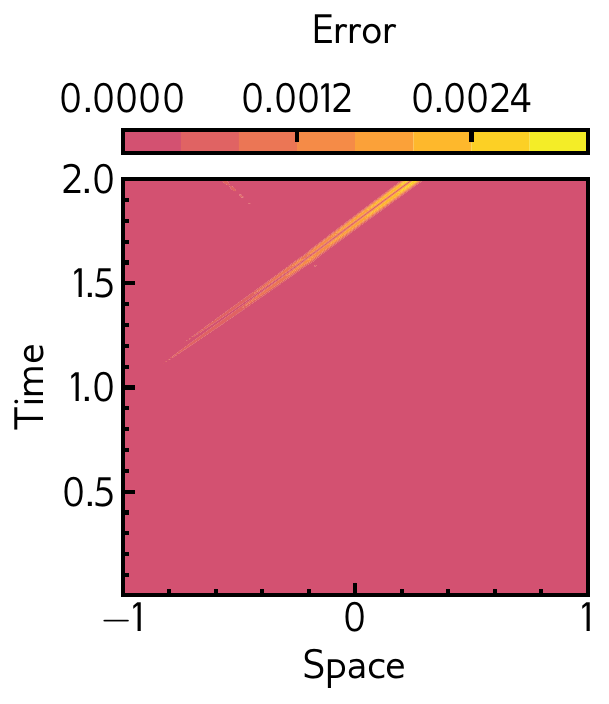}
	\end{subfigure}
	\begin{subfigure}[b]{0.31\textwidth}
			\centering
			\includegraphics[width=0.85\textwidth, trim=0 0 0 40, clip]{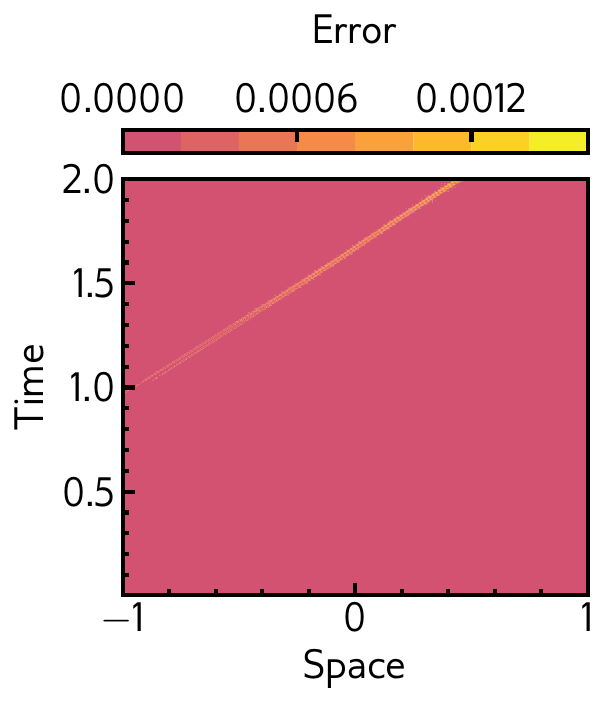}
	\end{subfigure}
	\caption{Shallow water equations' fluid velocity: The SDE-ActLearn-POD-KSNN solution (shown in first row), and the relative solution error (shown in second row). The viscosity values $\nu$ going from left to right are in the following order: $\{5 \times 10^{-1}, 5 \times 10^{-2}, 5 \times 10^{-5}\}$. All these $\nu$ values and time instances are outside of the training set. The active learning process is carried out using the distance criterion $\hat{\mathcal{D}}_2$.} 
	\label{fig:swe-u-rom-sol-surf}
\end{figure}

\Cref{fig:swe-avg-test-error-in-time} shows the relative error behavior for the SDE-ActLearn-POD-KSNN ROM solution along the time domain. More precisely, the reported error values at any time instance correspond to the $2$-norm of the solution error, normalized by the $2$-norm of the true solution at that time instance. The ROM is queried at a test time grid comprising of $499$ time steps starting from $0.004$, with a step size of $0.004$. Moreover, the error curves correspond to the average relative error over all five parameter test samples. It is clear from \Cref{fig:swe-avg-test-error-in-time} that the ROM solution error is bounded from above by the user-defined error tolerance of $10^{-2}$, rendering an accurate active-learning-driven ROM.

The computational runtime for the FOM solver and the SDE-ActLearn-POD-KSNN ROM is given in \Cref{tab:swe-timings}. The numerical experiments are undertaken on a system with an Intel{\small\textsuperscript{\textregistered}} Core{\texttrademark} i7-11800H @ 2.30GHz central processing unit and 32 GB of RAM. The creation of parameter-specific POD subspaces, the numerous distance evaluations during the active learning procedure, and the error estimator construction during the final stage of active learning, amount to $0.3928$ seconds, reported as SDE-AL in the table. Whereas, the time taken for the construction of $\mu$-KSNN in the offline phase is $1.0053$ seconds, which is reported under `Remaining' in the table. As the online query time for the ROM at a new parameter sample is $0.0122$, we observe a computational speedup of greater than $\mathcal{O}(10^{4})$ in comparison to the FOM solver query time. Moreover, we would like to highlight that the offline phase cost associated to the SDE-ActLearn-POD-KSNN ROM is significantly less than our previously proposed ActLearn-POD-KSNN ROM~\cite{morKapFB24}, which is built using an error-estimator-based active learning strategy. This further exhibits the advantage of our new SDE-AL framework.

\begin{figure}[!t]
	\centering
	
	\includegraphics[width=0.51\textwidth, trim=0 0 0 0, clip]{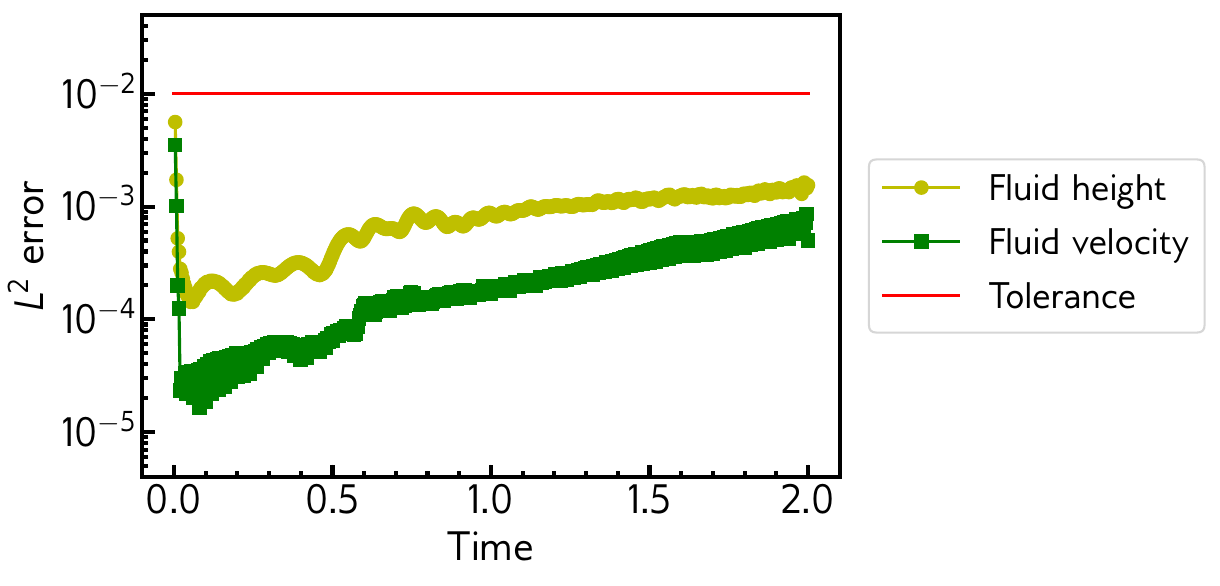}
	
	\caption{Shallow water equations: The relative SDE-ActLearn-POD-KSNN ROM error across the test time grid. The curves report the average error behavior for all five test samples in the parameter domain. The tolerance curve represents the $tol_e$ value appearing in \Cref{alg:sde-actlearn-B}.}
	
	\label{fig:swe-avg-test-error-in-time}
\end{figure}

\begin{table}[!h]  
	\begin{center}
		\begingroup
		\setlength{\tabcolsep}{10pt} 
		\renewcommand{\arraystretch}{1.5} 
		\begin{tabular}{| c || c | c | c | c |}
			\cline{1-5}
			\multirow{3}{*}{\begin{tabular}{@{}c@{}}FOM \\ solver\end{tabular}} & \multicolumn{4}{c|}{SDE-ActLearn-POD-KSNN ROM} \\ \cline{2-5}
			& \multicolumn{3}{c|}{Offline phase} & \multirow{2}{*}{\begin{tabular}{@{}c@{}}Online phase\end{tabular}} \\ 
			\cline{2-4}
			& SDE-AL & FOM queries & Remaining & \\
			\hline 
			\qty{249.56}{\second} & \qty{0.3928}{\second} & \num{19} $\times$ \qty{249.56}{\second} & \qty{1.0053}{\second} & \qty{0.0122}{\second} \\
			\hline
			\multicolumn{1}{c|}{} & \multicolumn{3}{c|}{Total: \qty{4743.03}{\second}} \\
			\cline{2-4}
		\end{tabular}
		\endgroup
		\caption{Shallow water equations: Runtime of the FOM and the SDE-ActLearn-POD-KSNN ROM constructed using the distance criterion $\hat{\mathcal{D}}_2$. The offline phase runtime for the ROM construction is reported in three parts: SDE-AL, FOM queries, and Remaining. The SDE-AL time corresponds to the time taken for creating the parameter-specific POD subspaces, along with all distance evaluations during the active learning iterations and the error estimator construction in the final stage. The time corresponding to FOM queries also includes the runtime of FOM simulations corresponding to the initial parameter set. The time required for all remaining computations in the offline phase is reported under `Remaining'.} %
		\label{tab:swe-timings}
	\end{center}
\end{table}

\FloatBarrier


\subsection{Incompressible Navier-Stokes Equations}%
\label{subsec:nse}

The incompressible Navier-Stokes equations are used to model the behavior of viscous fluids. The following form of the $2D$ Navier-Stokes equations is considered:
\begin{align*}
	\rho \left( \partial_t \mathbf{u} + \mathbf{u} \cdot \nabla \mathbf{u} \right) &= \mu \nabla \cdot \left(\nabla \mathbf{u} + (\nabla \mathbf{u})^T\right) - \nabla p,
	\\
	\nabla \cdot \mathbf{u} &= 0,
\end{align*}
where $\mathbf{u}(t,\mathbf{z}) \in \mathbb{R}^2$ is the velocity vector and $p(t,\mathbf{z}) \in \mathbb{R}$ is the fluid pressure, with time $t \ge 0$ and space $\mathbf{z}:= (x \times y) \in \Omega \subset \mathbb{R}^2$. The density and dynamic viscosity of the fluid are denoted by $\rho$ and $\mu$, respectively.

We consider a problem setup similar to the flow past a cylinder test case $2D$-$2$ from \cite{dfg-benchmark-ref}. A schematic of the spatial domain is shown in \Cref{fig:nse-setup}. The domain has a fluid inlet $\Gamma_{in}$ on the left boundary, a fluid outlet $\Gamma_{out}$ on the right boundary, and channel walls on the top and bottom boundaries, i.e., $\Gamma_{top}$ and $\Gamma_{bottom}$, respectively. The circular obstacle is shown in gray color in \Cref{fig:nse-setup} with its boundary denoted by $\Gamma_{cyl}$. Furthermore, the part of the domain where the fluid can flow is represented by $\Omega:=\Omega_{fluid}$. The exact form of the boundary conditions are given below:
\begin{align*}
	\mathbf{u} &= \left[ \frac{4 U y (0.41-y)}{0.41^2}, 0 \right]^\top, & \mathbf{z} \in \Gamma_{in},\\
	\mu \left( \nabla \mathbf{u} + (\nabla \mathbf{u})^\top\right) \cdot \mathbf{n} &= p \mathbf{n}, & \mathbf{z} \in \Gamma_{out},\\ 
	\mathbf{u} &= 0, & \mathbf{z} \in \Gamma_{cyl} \cup \Gamma_{top} \cup \Gamma_{bottom},
\end{align*}
where $U=1.5$ represents the maximum magnitude of the fluid velocity at the inlet and $\mathbf{n}$ is the outward normal vector at the boundary $\Gamma_{out}$. The initial velocity and pressure in $\Omega_{fluid}$ are taken to be zero, i.e., $\mathbf{u}(0, \mathbf{z}) = 0$, $p(0, \mathbf{z}) = 0$.

\begin{figure}[!t] 
	\centering
	
	\includegraphics[width=0.85\textwidth, trim=0 0 0 0, clip]{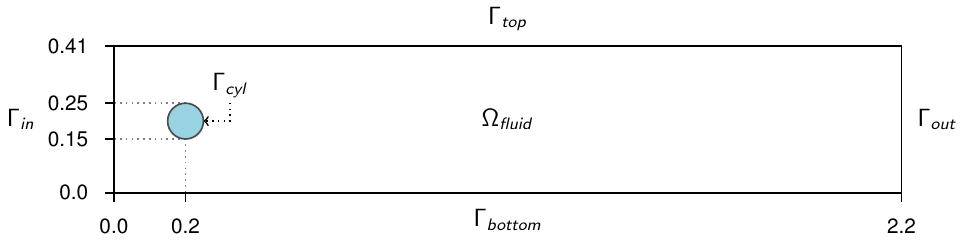}
	
	\caption{A representation of the spatial domain for the Navier-Stokes equations problem.}
	\label{fig:nse-setup}
\end{figure}

The domain $\Omega_{fluid}$ is discretized with tetrahedral mesh elements amounting to a total of \num{37514} grid nodes. The Navier-Stokes equations are numerically solved using the FEniCS package~\cite{fenics}. More specifically, our solver employs the finite element method and an incremental pressure correction scheme to handle the pressure-velocity coupling. For further details about the implementation, please refer to~\cite{morKapFB24}. We fix the fluid density to a value of $\rho = 1 \ kg/m^3$ and can freely vary the dynamic viscosity $\mu$, resulting in a variable Reynolds number between $[80, 200]$. In this work, we focus on constructing an actively-learned ROM for the fluid velocity magnitude over the complete domain $\Omega_{fluid}$.

\begin{figure}[!b] 
	\centering
	\begin{subfigure}[b]{0.35\textwidth}
		\centering
		\includegraphics[width=1.0\textwidth, trim=0 0 0 0, clip]{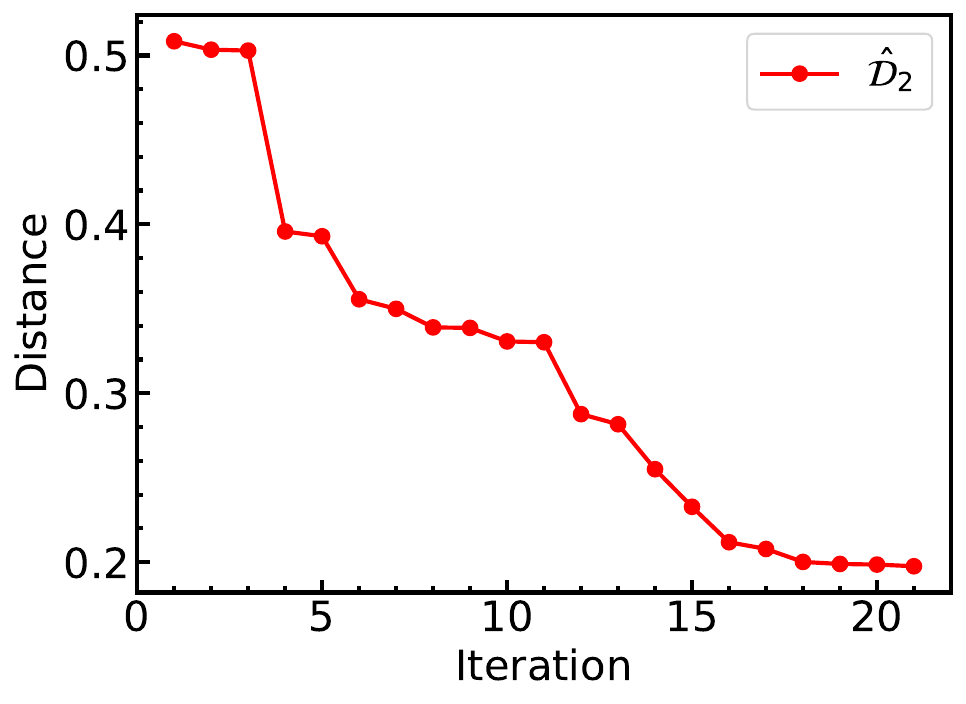}
		\caption{}
		\label{fig:nse-active-learn-a}
	\end{subfigure}
	\hspace{2em}
	\begin{subfigure}[b]{0.35\textwidth}
		\centering
		\includegraphics[width=1.0\textwidth, trim=0 0 0 0, clip]{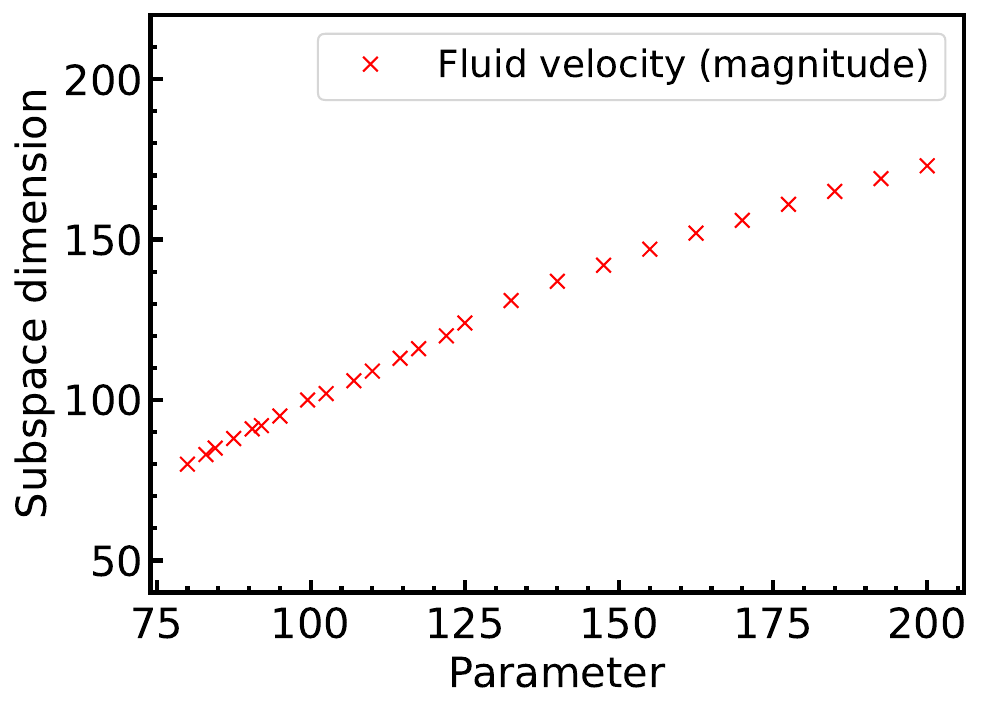}
		\caption{}
		\label{fig:nse-active-learn-b}
	\end{subfigure}
	\caption{Navier-Stokes equations: Subspace-distance-enabled active learning procedure with distance criterion $\hat{\mathcal{D}}_2$. (a)~shows the variation in the maximal distance value  $(\hat{\mathcal{D}}_2)_{max}^{(iter)}$ with iterations. For better readability, this is denoted by $\hat{\mathcal{D}}_2$ in the legend. (b)~shows the selected parameter samples along with the dimension of their respective POD subspace.} 
	\label{fig:nse-active-learn}
\end{figure}

To begin the active learning procedure, $81$ evenly distributed $Re$ values are considered from $[80, 200]$. Out of these, five $Re$ values, i.e., $\{80, 110, 140, 170, 200\}$, are initially taken in the parameter set $P$, whereas, the remaining $76$ $Re$ values are placed in the candidate set $P^*$. Furthermore, parameter-specific snapshot matrices for the fluid magnitude are constructed by querying the FOM solver corresponding to each of the $Re$ values from $P$. Similar to the shallow water equations example, parameter-specific subspaces are then formed, retaining $99.9999\%$ of the total energy, which corresponds to $\eta = 10^{-6}$.

With this problem, we intend to demonstrate the applicability of \Cref{alg:sde-actlearn-A} and terminate the active learning procedure as soon as a specific user-defined computational budget for FOM queries is satisfied. Moreover, the SDE-ActLearn-POD-NN ROM is constructed as outlined in \Cref{alg:SDE-ActLearn-POD-NN-offline-Variant-A}, where a deep NN is used to learn the POD coefficients. Once the initial set of parameter-specific subspaces are formed, the distances between all the neighboring subspaces are evaluated and the parameter pair with the most distant subspaces is identified. This is followed by the selection of a new sample from the candidate set $P^*$ as outlined in \Cref{alg:sde-actlearn-A}. In an iterative fashion, this procedure is continued, by updating the neighboring parameter pairs, evaluating the distance between new pairs, and selecting new samples from $P^*$ in parameter regions surrounded by subspaces that are farthest away.

\begin{figure}[!t]
	\centering
	\begin{subfigure}[b]{0.49\textwidth}
		\centering
		\includegraphics[width=1.0\textwidth, trim=0 0 0 0, clip]{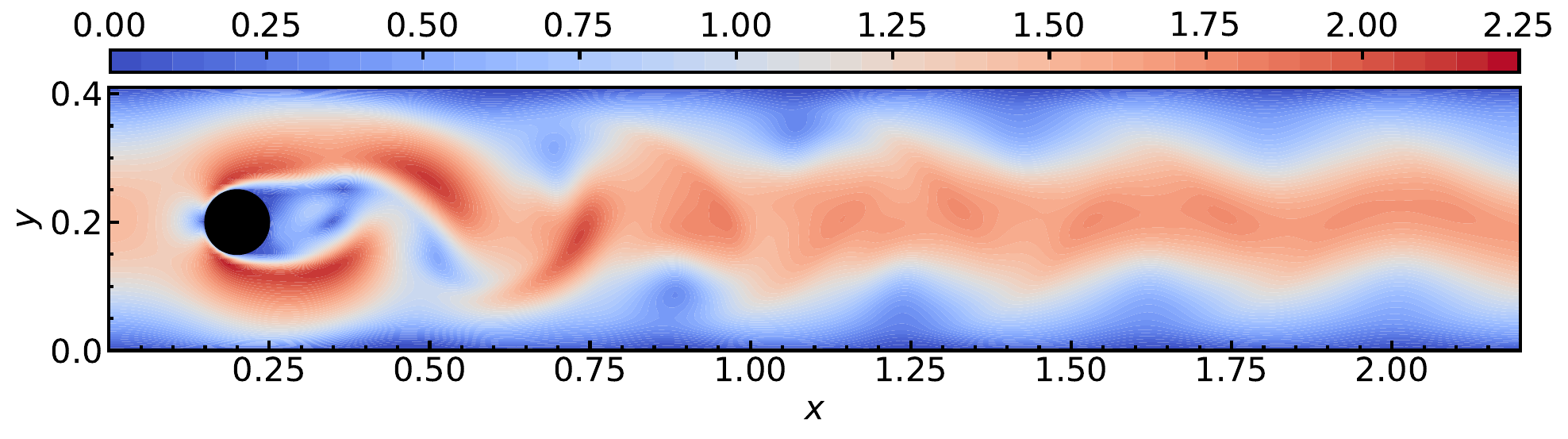}
	\end{subfigure}
	\begin{subfigure}[b]{0.49\textwidth}
		\centering
		\includegraphics[width=1.0\textwidth, trim=0 0 0 0, clip]{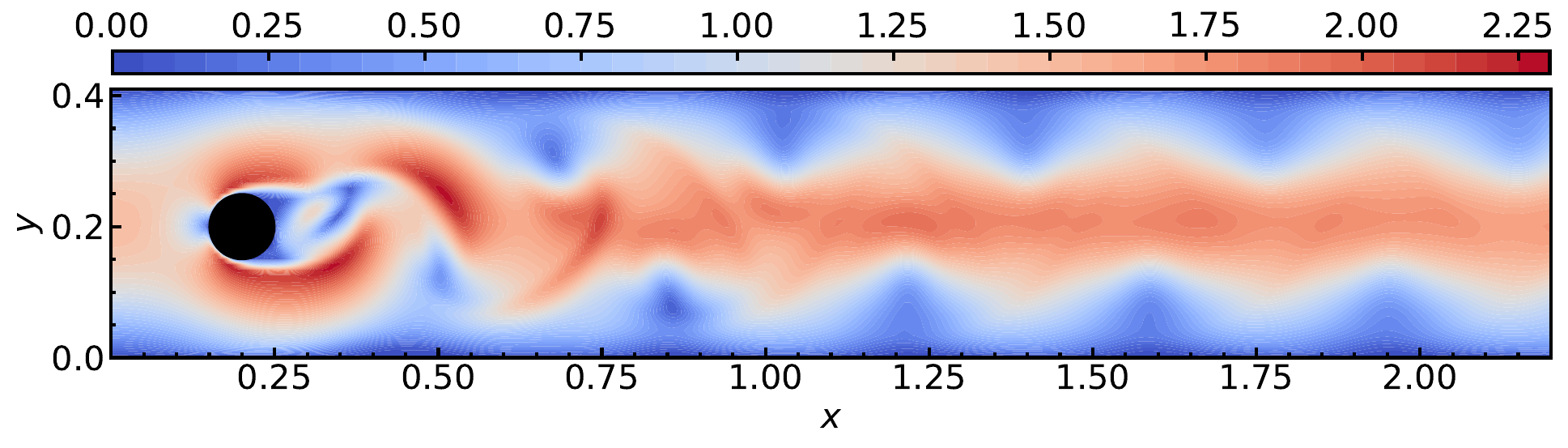}
	\end{subfigure}
	\begin{subfigure}[b]{0.49\textwidth}
		\centering
		\includegraphics[width=1.0\textwidth, trim=0 0 0 0, clip]{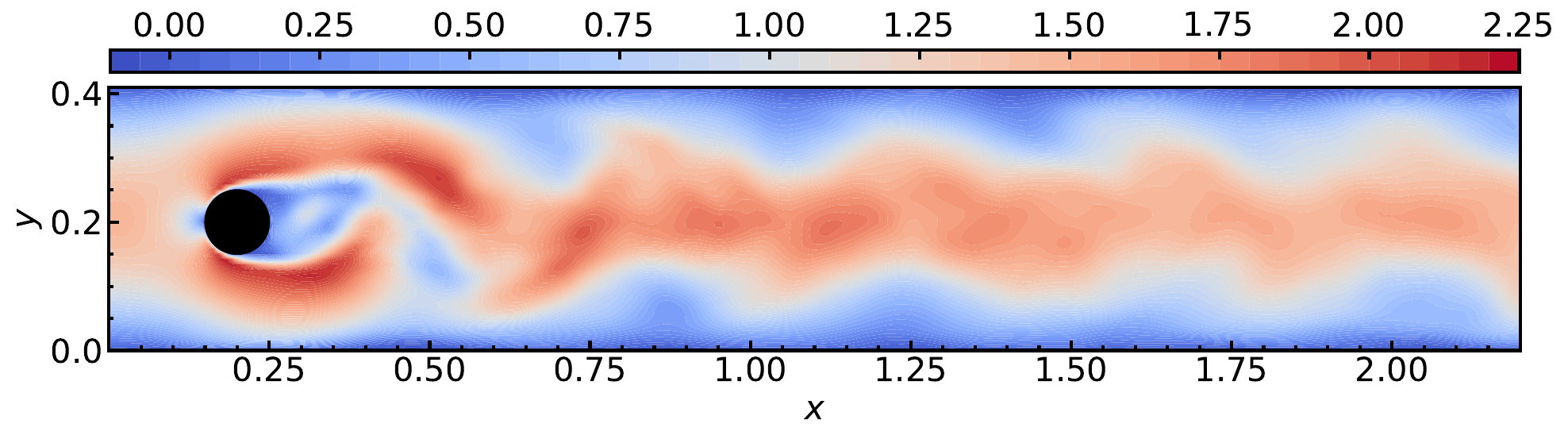}
	\end{subfigure}
	\begin{subfigure}[b]{0.49\textwidth}
		\centering
		\includegraphics[width=1.0\textwidth, trim=0 0 0 0, clip]{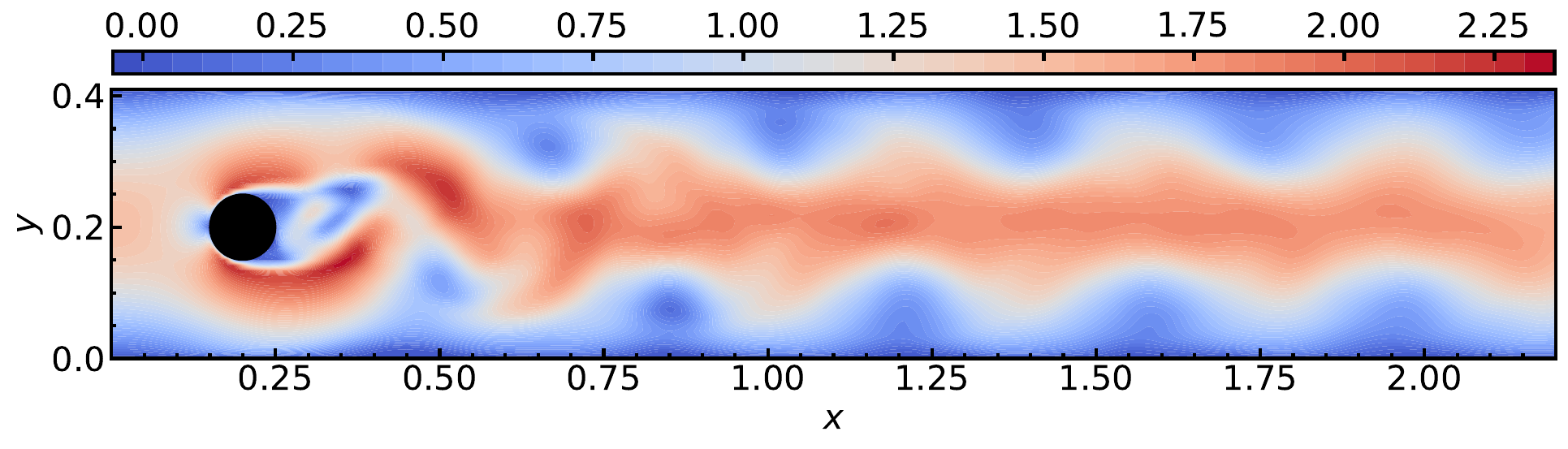}
	\end{subfigure}
	\begin{subfigure}[b]{0.49\textwidth}
		\centering
		\includegraphics[width=1.0\textwidth, trim=0 0 0 0, clip]{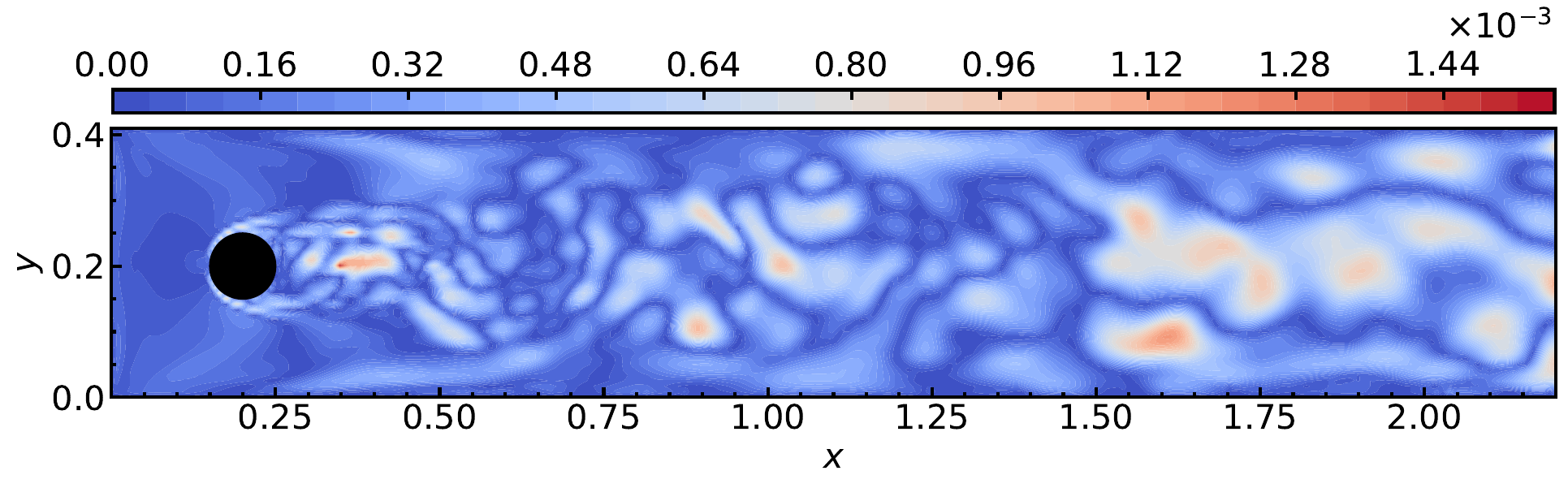}
	\end{subfigure}
	\begin{subfigure}[b]{0.49\textwidth}
		\centering
		\includegraphics[width=1.0\textwidth, trim=0 0 0 0, clip]{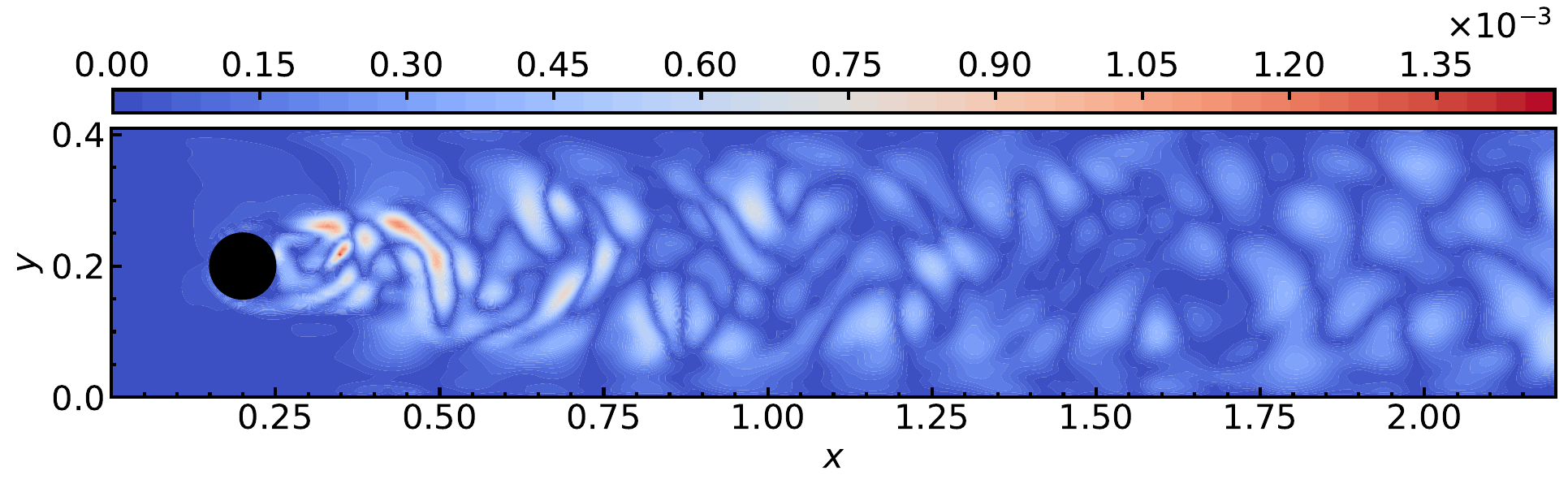}
	\end{subfigure}
	\caption{Navier-Stokes equations: The true solution (shown in first row), the SDE-ActLearn-POD-NN solution (shown in second row), and the relative solution error (shown in third row) at a representative test time instance $t=3.512$. The test $Re$ values in the left and right columns are $115$ and $180$, respectively.}
	\label{fig:nse-rom-sol}
\end{figure}

\begin{table}[!t]
	\begin{center}
		\begingroup
		\setlength{\tabcolsep}{10pt} 
		\renewcommand{\arraystretch}{1.5} 
		\begin{tabular}{| c || c | c | c | c |}
			\cline{1-5}
			\multirow{3}{*}{\begin{tabular}{@{}c@{}}FOM \\ solver\end{tabular}} & \multicolumn{4}{c|}{SDE-ActLearn-POD-NN ROM} \\ \cline{2-5}
			& \multicolumn{3}{c|}{Offline phase} & \multirow{2}{*}{\begin{tabular}{@{}c@{}}Online phase\end{tabular}} \\ 
			\cline{2-4}
			& SDE-AL & FOM queries & Remaining & \\
			\hline 
			\qty{38742}{\second} & \qty{59.8689}{\second} & \num{25} $\times$ \qty{38742.18}{s} & \qty{2117.5016}{\second} & \qty{0.0043}{\second} \\
			\hline
			\multicolumn{1}{c|}{} & \multicolumn{3}{c|}{Total: \qty{970732}{\second}} \\ 
			\cline{2-4}
		\end{tabular}
		\endgroup
		\caption{Navier-Stokes equations: Runtime of the FOM and that of the SDE-ActLearn-POD-NN ROM. The offline phase runtime for the ROM construction is reported in three parts: SDE-AL, FOM queries, and Remaining. The SDE-AL time corresponds to the time taken for creating the parameter-specific POD subspaces along with all distance evaluations during the active learning iterations. The time corresponding to FOM queries also includes the runtime of FOM simulations corresponding to the initial parameter set. The time required for all remaining computations in the offline phase is reported under `Remaining'.} %
		\label{tab:nse-timings}
	\end{center}
\end{table}

Since for the problem of shallow water equations, we already demonstrated the benefit of employing the distance measure $\hat{\mathcal{D}}_2$ \cref{eqn:distance-2-actlearn} during the SDE-AL process, we only employ $\hat{\mathcal{D}}_2$ in the SDE-AL procedure for this example. \Cref{fig:nse-active-learn-a} shows that the maximal distance monotonically decreases as the iterations proceed. With the computational budget set to $20$ queries of the FOM solver, the process terminates with the maximum distance between the subspaces being just under $0.2$. All the chosen parameter samples, along with their respective subspace dimensions, are shown in \Cref{fig:nse-active-learn-b}. We observe that lower values of $Re$ are slightly more preferred by the active learning process. This can be understood to be caused by a relatively greater variation in the solution features in this regime of the parameter space.
    
The SDE-ActLearn-POD-NN ROM is queried at the following out-of-training $Re$ values: $\{115,\allowbreak 145, 180\}$. \Cref{fig:nse-rom-sol} shows the fluid velocity predictions at two representative test parameter samples, corresponding to a representative time instance $t = 3.512$. The respective relative solution errors and the true FOM solutions are also shown. The reported relative error values are the absolute of the difference between the FOM and ROM solutions, normalized by the $2$-norm of the FOM solution. The reduced basis $\mathbb{V}$ is formed by retaining $99.9\%$ of the total energy, corresponding to $\hat{\eta} = 10^{-3}$, resulting in the reduced basis size of $174$. Upon assessing \Cref{fig:nse-rom-sol}, we can see that the actively built ROM is able to generate accurate solution profiles, successfully capturing the vortex shedding phenomenon in the cylinder wake.

\Cref{tab:nse-timings} provides runtime for the FOM solver and the SDE-ActLearn-POD-NN ROM. The FOM solver is executed on a node of the compute cluster Mechthild available at the Max Planck Institute Magdeburg. The node is equipped with two Intel{\small\textsuperscript{\textregistered}} Xeon Silver 4110 central processing units and $192$ GB of RAM. The deep NN used for this example is trained on an NVIDIA Tesla P100 graphics processing unit associated with a similar cluster node. Like the shallow water equations example, we report the time required for creating the parameter-specific subspaces along with all the distance evaluations during the active learning process under SDE-AL in \Cref{tab:nse-timings}. This amounts to $59.8689$ seconds. The runtime reported under `Remaining' corresponds to the global reduced basis construction in \cref{eqn:pod-nn-svd-2,eqn:pod-nn-bases-2}, NN data preparation, as well as the NN model construction and training. Out of this, the network training takes $2072.5892$ seconds. Overall, the offline phase cost is dominated by computations towards the FOM queries. The SDE-AL framework enables selection of critical parameter locations in the offline phase, thereby, rendering an effective ROM construction. In comparison with a FOM solver query, the ROM's online query time of $0.0043$ seconds provides us a computational speedup of greater than $\mathcal{O}(10^6).$

\FloatBarrier

\section{Conclusions}%
\label{sec:conclude}

We present a novel general-purpose active learning strategy for non-intrusive reduced-order modeling of parametric dynamical systems. The selection of critical parameter locations during the offline phase of constructing ROMs is driven by comparing the similarity between parameter-specific subspaces and iteratively picking new samples from regions of most dissimilar parametric solution fields. This is accomplished by developing a distance metric for evaluating the distance between linear subspaces of different sizes. Due to the distance metric being computationally inexpensive, it facilitates repeated distance evaluations between numerous subspace pairs efficiently, thereby, acting as the central pillar supporting the proposed subspace-distance-enabled active learning (SDE-AL) framework. 

Our ideas are successfully tested on two challenging fluid flow problems with competing convective and diffusive phenomena, where multiple shock profiles interact along the spatial domain in one problem, and in another problem, a flow scenario features a von K\'arm\'an vortex street with varying vortex shedding frequencies across the parameter domain. The results demonstrate the advantage of using the proposed distance metric for SDE-AL. Particularly, the maximal distance between the parameter-specific subspaces at each active learning iteration is observed to monotonically decrease as new parameter samples are progressively selected. Moreover, we notice that the sampling pattern in the parameter domain is robust to the level of truncation used for forming the parameter-specific subspaces. As long as the subspaces are expressive enough to reasonably describe the solution dynamics, we can reliably sample critical parameter locations without needing to worry about the level of information retained in the parametric solution subspaces.

Both variants of SDE-AL---either termination after meeting a user-defined computational budget for high-fidelity solver queries or termination upon sufficing user-defined distance and error tolerances---exhibit good performance with accurate ROM constructions. Moreover, the SDE-AL strategy is flexible and, in principle, can be applied to any reduced-order modeling technique. We showcased its application to two non-intrusive reduced-order modeling methods, namely POD-KSNN and POD-NN, proposing their efficient active-learning-driven counterparts, i.e., SDE-ActLearn-POD-KSNN and SDE-ActLearn-POD-NN ROMs. Compared to the high-fidelity solvers, across both problems, the ROMs constructed using SDE-AL render a considerable computational speedup for solution predictions at new parameter locations. More importantly, the SDE-AL framework improves the offline ROM construction phase by effectively limiting the number of parametric high-fidelity solution snapshots required. Hence, the proposed active learning strategy is a promising candidate for augmenting non-intrusive ROMs intended for parametric many-query applications that require an effective generation of solution dynamics across numerous locations in the parameter domain.


\section*{Code and Data Availability}
\addcontentsline{toc}{section}{Code and Data Availability}

The source code and data utilized for the numerical experiments are available for review upon request and will be made publicly available upon publication of this manuscript. %


\section*{Acknowledgments}%
\addcontentsline{toc}{section}{Acknowledgments}

Harshit Kapadia is supported by the International Max
Planck Research School for Advanced Methods in Process
and Systems Engineering (IMPRS-ProEng). This work was also supported by the research initiative “SmartProSys: Intelligent Process Systems for the Sustainable Production of Chemicals” funded by the Ministry for Science, Energy, Climate Protection and the Environment of the State of Saxony-Anhalt.

%
%

\appendix

\FloatBarrier

\section{Implementation Details} %
\label{appendix:implementation-details}

We provide details about the KSNN setups used for the shallow water equations example and the NN setup used for the Navier Stokes equations example in \Cref{tab:implementation-details}. Instead of training for locally-adaptive kernel widths using the ADSIT procedure from~\cite{morKapFB24}, we specify a uniform width for all the activation kernels in our KSNN constructions, which is denoted by $\epsilon$ in the table. To train the deep feed-forward NN, the learning rate is initially specified as $0.01$, which is subsequently halved at every $1000$ epochs. The Adam optimizer~\cite{adam-optimizer} is used to train the network. 
\begin{table}[!h] 
	\begin{center}
		\begingroup
		\setlength{\tabcolsep}{5pt} 
		\renewcommand{\arraystretch}{1.5} 
		\begin{tabular}{|c||c|c|c|c|c|} 
			\cline{1-6}
			Model $\&$ Network & Activation & Epochs & $\epsilon$ & \# layers & Layer width \\
			\cline{1-6}
			\noalign{\vskip\doublerulesep\vskip-\arrayrulewidth} 
			\cline{1-6}
			SWE: $\mu$-KSNN & Multi-quadric & $-$ & $10^{-3}$ & $3$ & $19$ \\
			\cline{1-6}
			SWE: $t$-KSNN & Multi-quadric & $-$ & $10^{-3}$ & $3$ & $200$ \\
			\cline{1-6}
			NSE: NN & ReLU & $3000$ & $-$ & $6$ & $300$ \\
			\cline{1-6}
		\end{tabular}
		\endgroup
		\caption{KSNN and NN implementation details. The numerical examples are denoted as follows: SWE: Shallow water equations; NSE: Navier-Stokes equations. Note that the number of layers reported include the input and output layers. The reported layer widths correspond to the widths of the hidden layers in the respective networks.}  
		\label{tab:implementation-details}
	\end{center}
\end{table}


\addcontentsline{toc}{section}{References}
\bibliographystyle{plainurl}
\bibliography{journals,refs2}
  
\end{document}